\documentclass[12pt,reqno]{amsart}

\usepackage{amssymb,latexsym}
\usepackage{mathrsfs}
\usepackage{a4wide}
\usepackage{amsthm}
\usepackage[v2,tips]{xy}
\usepackage[T1]{fontenc}
\usepackage[latin1]{inputenc}
\usepackage{tikz}
\usetikzlibrary{matrix,arrows}

\newtheorem{theorem}{Theorem}[section]
\newtheorem{lemma}[theorem]{Lemma}
\newtheorem{proposition}[theorem]{Proposition}
\newtheorem{corollary}[theorem]{Corollary}

\theoremstyle{definition}
\newtheorem{example}[theorem]{Example}
\newtheorem{remark}[theorem]{Remark}

\newtheorem{definition}[theorem]{Definition}

\numberwithin{equation}{section}

\newcounter{smallromans}

\newenvironment{romanenumerate}
{\begin{list}{{\normalfont\textrm{(\roman{smallromans})}}}%
  {\usecounter{smallromans}\setlength{\itemindent}{0cm}%
   \setlength{\leftmargin}{5.5ex}\setlength{\labelwidth}{5.5ex}%
   \setlength{\topsep}{.5ex}\setlength{\partopsep}{.5ex}%
   \setlength{\itemsep}{0.1ex}}}%
{\end{list}}

\newcommand{\romanref}[1]{{\normalfont\textrm{(\ref{#1})}}}

\newcounter{smallromansdash}

{\end{list}}

\newcounter{bigromans} 
\newenvironment{capromanenumerate}
{\begin{list}{{\normalfont\textrm{(\Roman{bigromans})}}}%
    {\usecounter{bigromans}\setlength{\itemindent}{0cm}%
      \setlength{\leftmargin}{5.5ex}\setlength{\labelwidth}{6ex}%
      \setlength{\topsep}{.5ex}\setlength{\partopsep}{.5ex}%
      \setlength{\itemsep}{0.1ex}}}%
  {\end{list}}

\newcounter{smallarabics}
{\end{list}}

\newcounter{smallalphs}

\newenvironment{alphenumerate}
{\begin{list}{{\normalfont\textrm{(\alph{smallalphs})}}}%
  {\usecounter{smallalphs}\setlength{\itemindent}{0cm}%
   \setlength{\leftmargin}{5.5ex}\setlength{\labelwidth}{5.5ex}%
   \setlength{\topsep}{.5ex}\setlength{\partopsep}{.5ex}%
   \setlength{\itemsep}{0.1ex}}}%
{\end{list}}

\newcommand{\alphref}[1]{{\normalfont\textrm{(\ref{#1})}}}

\renewcommand{\le}{\ensuremath{\leqslant}}

\renewcommand{\ge}{\ensuremath{\geqslant}}

\newcommand{\N}{\mathbb{N}}

\newcommand{\R}{\mathbb{R}}
\newcommand{\C}{\mathbb{C}}

\newcommand{\realpart}{\operatorname{Re}}

\newcommand{\smashw}[2][l]{{\text{\makebox[0pt][#1]{$#2$}}}}

\renewcommand{\epsilon}{\ensuremath{\varepsilon}}
\renewcommand{\phi}{\ensuremath{\varphi}}

\title[A topological dichotomy with applications]{A
  weak$^*$-topological dichotomy\\ with applications in operator theory}

\author[Kania]{Tomasz Kania}\email{t.kania@lancaster.ac.uk}%
\address{Department of Mathematics and Statistics, Fylde College,
  Lancaster University, Lancaster LA1 4YF, United Kingdom.}

\author[Koszmider]{Piotr Koszmider}\email{P.Koszmider@Impan.pl}%
\address{Institute of Mathematics, Polish Academy of Sciences,
  ul. \'Sniadeckich 8, 00-956 War\-sza\-wa, Poland.}

\author[Laustsen]{Niels Jakob
  Laustsen}\email{n.laustsen@lancaster.ac.uk}%
\address{Department of Mathematics and Statistics, Fylde College,
  Lancaster University, Lancaster LA1 4YF, United Kingdom.}

\date{} 
\begin{document}

\begin{abstract} 
  Denote by~$[0,\omega_1)$ the locally compact Hausdorff space
  consisting of all countable ordinals, equipped with the order
  topology, and let~$C_0[0,\omega_1)$ be the Banach space of
  scalar-valued, continuous functions which are defined
  on~$[0,\omega_1)$ and vanish eventually. We show that a weakly$^*$
  compact subset of the dual space of~$C_0[0,\omega_1)$ is either
  uniformly Eberlein compact, or it contains a homeomorphic copy of
  the ordinal interval~$[0,\omega_1]$.

  Using this result, we deduce that a Banach space which is a quotient
  of~$C_0[0,\omega_1)$ can either be embedded in a Hilbert-generated
  Banach space, or it is isomorphic to the direct sum
  of~$C_0[0,\omega_1)$ and a subspace of a Hilbert-generated Banach
  space.  Moreover, we obtain a list of eight equivalent conditions
  describing the Loy--Willis ideal, which is the unique maximal ideal
  of the Banach algebra of bounded, linear operators
  on~$C_0[0,\omega_1)$. As a consequence, we find that this ideal has
  a bounded left approximate identity, thus resolving a problem left
  open by Loy and Willis, and we give new proofs, in some cases of
  stronger versions, of several known results about the Banach
  space~$C_0[0,\omega_1)$ and the operators acting on it.
\end{abstract}

\maketitle

\section{Introduction and statement of main results}
\noindent
The main motivation behind this paper is the desire to deepen our
understanding of the Banach algebra $\mathscr{B}(C_0[0,\omega_1))$ of
(bounded, linear) operators on the Banach space $C_0[0,\omega_1)$ of
scalar-valued, continuous functions which are defined on the locally
compact ordinal in\-ter\-val~$[0,\omega_1)$ and vanish eventually.
Our strategy is to begin at a topological level, where we establish a
new dichotomy for weakly$^*$ compact subsets of the dual space
of~$C_0[0,\omega_1)$, and then use this dichotomy to obtain
information about~$C_0[0,\omega_1)$ and the operators acting on it,
notably a list of eight equivalent conditions characterizing the
unique maximal ideal of~$\mathscr{B}(C_0[0,\omega_1))$.

The Banach space $C_0[0,\omega_1)$ is of course isometrically
isomorphic to the hyperplane \[ \bigl\{f\in C[0,\omega_1] :
f(\omega_1)=0\bigr\} \] of the Banach space~$C[0,\omega_1]$ of
scalar-valued, continuous functions on the compact ordinal
interval~$[0,\omega_1]$, and hence $C_0[0,\omega_1)$
and~$C[0,\omega_1]$ are isomorphic. Since our focus is on properties
that are invariant under Banach-space isomorphism, we shall freely
move between these two spaces in the following summary of the history
of their study.

Semadeni~\cite{semadeniC} was the first to realize
that~$C[0,\omega_1]$ is an interesting Banach space, showing that it
is not isomorphic to its square, and thus producing the joint first
example of an infinite-dimensional Banach space with this
property. (The other example, due to Bessaga and
Pe{\l}czy{\a'n}ski~\cite{bpCartesian}, is James's quasi-reflexive
Banach space.) The Banach-space structure of~$C_0[0,\omega_1)$ was
subsequently explored in much more depth by Alspach and
Benyamini~\cite{alspachbenyamini}, whose main conclusion is
that~$C_0[0,\omega_1)$ is primary, in the sense that
whenever~$C_0[0,\omega_1)$ is decomposed into a direct sum of two
closed subspaces, one of these subspaces is necessarily isomorphic
to~$C_0[0,\omega_1)$.

Loy and Willis~\cite{lw} initiated the study of the Banach
algebra~$\mathscr{B}(C[0,\omega_1])$ from an auto\-matic-continuity
point of view, proving that each derivation from
$\mathscr{B}(C[0,\omega_1])$ into a Banach algebra is automatically
continuous. Their result was subsequently generalized by
Ogden~\cite{ogden}, who established the automatic continuity of each
algebra homomorphism from~$\mathscr{B}(C[0,\omega_1])$ into a Banach
algebra.

Loy and Willis's starting point is the clever identification of a
maximal ideal~$\mathscr{M}$ of co\-dimen\-sion one
in~$\mathscr{B}(C[0,\omega_1])$ (see
equation~\eqref{defnLoyWillisIdeal} below for details of the
definition), while their main technical step \cite[Theorem~3.5]{lw} is
the construction of a bounded right approx\-imate identity
in~$\mathscr{M}$.  The first- and third-named
authors~\cite{kanialaustsen} showed recently that~$\mathscr{M}$ is the
unique maximal ideal of~$\mathscr{B}(C[0,\omega_1])$, and named it the
\emph{Loy--Willis ideal}. We shall here give a new proof of this
result, together with several new characterizations of the Loy--Willis
ideal.  As a consequence, we obtain that~$\mathscr{M}$ has a bounded
left approximate identity, thus complementing Loy and Willis's key
result mentioned above.

The tools that we shall use come primarily from point-set topology and
Banach space theory, and several of our results may be of independent
interest to researchers in those areas, as well as to operator
theorists. Before entering into a more detailed description of this
paper, let us introduce four no\-tions that will play important roles
throughout.
\begin{itemize}
\item A topological space~$K$ is \emph{Eberlein compact} if it is
  homeomorphic to a weakly compact subset of a Banach space; and~$K$
  is \emph{uniformly Eberlein compact} if it is homeomorphic to a
  weakly compact subset of a Hilbert space.
\item A Banach space $X$ is \emph{weakly compactly generated} if it
  contains a weakly compact subset whose linear span is dense in~$X$;
  and~$X$ is \emph{Hilbert-generated} if there exists a a bounded
  operator from a Hilbert space onto a dense subspace of~$X$.
\end{itemize}
These notions are closely related. Uniform Eber\-lein com\-pact\-ness
clearly implies Eberlein compactness, and likewise Hilbert-generation
implies weakly compact generation. A much deeper result, due to Amir
and Lindenstrauss~\cite{al}, states that a compact space~$K$ is
Eberlein compact if and only if the Banach space~$C(K)$ is weakly
compactly generated; and a similar relationship holds between the
other two notions. Their relevance for our purposes stems primarily
from the fact that the ordinal interval~$[0,\omega_1]$ is one of the
``simplest'' compact spaces which is not Eberlein compact.

We shall now outline how this paper is organized and state its main
conclusions. Section~\ref{sectPrelim} contains details of our
notation, key elements of previous work, and some preliminary results.
In section~\ref{section3}, we proceed to study the weakly$^*$ compact
subsets of the dual space of $C_0[0,\omega_1)$, proving in particular
the following topological dichotomy.

\begin{theorem}[Topological Dichotomy]\label{dichotomy}
  Exactly one of the following two alternatives holds for each
  weakly$^*$ compact subset~$K$ of~$C_0[0,\omega_1)^*\colon$
  \begin{capromanenumerate}
  \item\label{dichotomy1} either $K$ is uniformly Eberlein compact; or
  \item\label{dichotomy2} $K$ contains a homeomorphic copy of
    $[0,\omega_1]$ of the form
    \[ \{\rho+\lambda\delta_\alpha: \alpha\in D\}\cup\{\rho\}, \]
    where $\rho\in C_0[0,\omega_1)^*$, $\lambda$ is a non-zero scalar,
      $\delta_\alpha$ is the point evaluation at~$\alpha$, and $D$ is
      a closed and unbounded subset of~$[0,\omega_1)$.
  \end{capromanenumerate}
\end{theorem}

In Section~\ref{section4}, we turn our attention to the structure of
operators acting on~$C_0[0,\omega_1)$.  In the case where~$T$ is a
bounded, linear surjection from~$C_0[0,\omega_1)$ onto an arbitrary
Banach space~$X$, the adjoint~$T^*$ of~$T$ induces a weak$^*$
homeomorphism of the unit ball of $X^*$ onto a bounded subset
of~$C_0[0,\omega_1)^*$, and hence the above topological dichotomy
leads to the following operator-theoretic dichotomy.

\begin{theorem}[Operator-theoretic Dichotomy]\label{surjective}
  Let~$X$ be a Banach space, and suppose that there exists a bounded,
  linear surjection~$T\colon C_0[0,\omega_1)\to X$. Then exactly one
    of the following two alternatives holds:
  \begin{capromanenumerate}
  \item\label{surjective1} either $X$ embeds in a Hilbert-generated
    Banach space; or
  \item\label{surjective2} the identity operator on~$C_0[0,\omega_1)$
    factors through~$T$, and $X$ is isomorphic to the direct sum
    of~$C_0[0,\omega_1)$ and a subspace of a Hilbert-generated Banach
    space.
  \end{capromanenumerate}
\end{theorem}

As another consequence of Theorem~\ref{dichotomy}, we obtain the
following result.
\begin{theorem}\label{clubandscalar}
  For each bounded operator~$T$ on~$C_0[0,\omega_1)$, there exist a
    unique scalar~$\phi(T)$ and a closed and unbounded subset~$D$
    of~$[0,\omega_1)$ such that
  \begin{equation}\label{clubandscalarEq1}
    (Tf)(\alpha)=\phi(T) f(\alpha)\qquad (f\in
    C_0[0,\omega_1),\,\alpha\in D). \end{equation} Moreover, the
  mapping $T\mapsto\phi(T)$ is linear and multiplicative, and thus a
  character on the Banach algebra~$\mathscr{B}(C_0[0,\omega_1))$.
\end{theorem}
We shall call~$\phi$ the \emph{Alspach--Benyamini character} because,
after having discovered the above theorem, we learnt that its main
part can be found in \cite[p.~76, line~$-6$]{alspachbenyamini}.

Theorem~\ref{clubandscalar} is the key step towards our main result: a
list of eight equivalent ways of describing the Loy--Willis
ideal~$\mathscr{M}$ of~$\mathscr{B}(C_0[0,\omega_1))$.  Before we can
state it, another piece of notation is required. The set
\begin{equation}\label{DefnL00} L_0 = \bigcup_{\alpha<\omega_1}
  [0, \alpha]\times\{\alpha+1\} \end{equation} is a
locally compact Hausdorff space with respect to the topology inherited
from the product topology on~$[0,\omega_1)^2$, and, as we shall see
in Corollary~\ref{L0uniformE} below, the Banach space~$C_0(L_0)$ is 
Hilbert-generated. 

\begin{theorem}\label{thmcharloywillis}
  The following eight conditions are equivalent for each bounded
  operator~$T$ on $C_0[0,\omega_1)\colon$
  \begin{alphenumerate}
  \item\label{TinLW} $T$ belongs to the Loy--Willis ideal
    $\mathscr{M};$
  \item \label{lambdazero} there is a closed and unbounded subset~$D$
    of~$[0,\omega_1)$ such that $T_{\alpha,\alpha} = 0$ for each
    $\alpha\in D$, where~$T_{\alpha,\alpha}$ denotes the
    $\alpha^{\text{th}}$ diagonal entry of the matrix associated
    with~$T$, as defined in equation~\eqref{Rudinmatrix} below;
  \item\label{character0} $T$ belongs to the kernel of the
    Alspach--Benyamini character~$\phi;$
  \item\label{Lzero} $T$ factors through the Banach space~$C_0(L_0);$
  \item\label{tishg} the range of $T$ is contained in a
    Hilbert-generated subspace of~$C_0[0,\omega_1);$
 \item\label{tiswcg} the range of $T$ is contained in a weakly
   compactly generated subspace of~$C_0[0,\omega_1);$
  \item\label{Tdoesntfix} $T$ does not fix a copy
    of~$C_0[0,\omega_1);$
  \item\label{Identity} the identity operator on $C_0[0,\omega_1)$
    does not factor through~$T$.
\end{alphenumerate}
\end{theorem}

\begin{remark}\label{remarkIndepProof}
\begin{romanenumerate}
\item\label{remarkIndepProof1} The equivalence of
  conditions~\alphref{TinLW} and~\alphref{Identity} of
  Theorem~\ref{thmcharloywillis} is the main result of a recent paper
  by the first- and third-named authors~\cite{kanialaustsen}. The
  proof that we shall give of Theorem~\ref{thmcharloywillis} will not
  depend on that result, and thus provides an alternative proof of it.
\item\label{remarkIndepProof2} The equivalence of
  conditions~\alphref{TinLW} and~\alphref{Lzero} of
  Theorem~\ref{thmcharloywillis} disproves the conjecture stated
  immediately after \cite[equation~(5.4)]{kanialaustsen}.
\end{romanenumerate}
\end{remark}

Theorem~\ref{thmcharloywillis} has a number of interesting
consequences, as we shall now explain.  The first, and arguably most
important, of these relies on the following notion.

\begin{definition} A net $(e_\gamma)_{\gamma\in\Gamma}$ in a Banach
  algebra~$\mathscr{A}$ is a \emph{bounded left approximate identity}
  if $\sup_{\gamma\in\Gamma}\|e_\gamma\|<\infty$ and the net
  $(e_\gamma a)_{\gamma\in\Gamma}$ converges to~$a$ for each
  $a\in\mathscr{A}$.  A \emph{bounded right approximate identity} is
  defined analogously, and a \emph{bounded two-sided approximate
    identity} is a net which is simultaneously a bounded left and
  right approximate identity.
\end{definition}

A well-known theorem of Dixon \cite[Proposi\-tion~4.1]{dixon} states
that a Banach algebra which has both a bounded left and a bounded
right approximate identity has a bounded two-sided approximate
identity.  As already mentioned, Loy and Willis constructed a bounded
right approximate identity in~$\mathscr{M}$. Although they did not
state it formally, their result immediately raises the question
whether~$\mathscr{M}$ contains a bounded left (and hence two-sided)
approximate identity. We can now provide a positive answer to this
question.
\begin{corollary}\label{blai} 
  The Loy--Willis ideal~$\mathscr{M}$ contains a net
  $(Q_D)_{D\in\Gamma}$ of projections, each having norm at most two,
  such that, for each operator~$T\in\mathscr{M}$, there is
  $D_0\in\Gamma$ for which \mbox{$Q_DT = T$} whenever $D\ge
  D_0$. Hence $(Q_D)_{D\in\Gamma}$ is a bounded left approximate
  identity in~$\mathscr{M}$.
\end{corollary}

When discovering this result, we were surprised that the net
$(Q_DT)_{D\in\Gamma}$ does not just converge to~$T$, but it 
actually equals~$T$ eventually. We have, however, subsequently
realized that the even stronger, two-sided counterpart of this
phenomenon occurs in the unique maximal ideal of the $C^*$-algebra
$\mathscr{B}(\ell_2(\omega_1))$, where $\ell_2(\omega_1)$ denotes the
first non-separable Hilbert space; see Example~\ref{BLAIexample} for
details.

Further consequences of Theorem~\ref{thmcharloywillis} include
generalizations of two classical Banach-space theoretic results, the
first of which is Semadeni's seminal observation~\cite{semadeniC} that
$C_0[0,\omega_1)$ is not isomorphic to its square.
\begin{corollary}\label{sum08082012}   
  Let $m,n\in\N$, and suppose that $C_0[0,\omega_1)^m$ is isomorphic
  to either a subspace or a quotient of~$C_0[0,\omega_1)^n$.  Then
  $m\le n$.
\end{corollary}

The other is Alspach and Benyamini's main theorem
\cite[Theorem~1]{alspachbenyamini} as it applies to
$C_0[0,\omega_1)$: this Banach space is
primary~\cite[Theorem~1]{alspachbenyamini}.

\begin{corollary}\label{cor08082012b}  For each bounded, linear
  projection~$P$ on~$C_0[0,\omega_1)$, either the kernel of~$P$ is
  isomorphic to $C_0[0,\omega_1)$ and the range of~$P$ embeds
  in~$C_0(L_0)$, or \emph{vice versa}.
\end{corollary}

Another Banach-space-theoretic consequence of
Theorem~\ref{thmcharloywillis} is as follows; it can alternatively be
deduced from \cite[Lemma~1.2 and Proposition~2]{alspachbenyamini}.

\begin{corollary}\label{cor26Oct2012a}
  Let~$X$ be a closed subspace of~$C_0[0,\omega_1)$ such that $X$ is
  isomorphic to $C_0[0,\omega_1)$. Then $X$ contains a closed subspace
  which is complemented in~$C_0[0,\omega_1)$ and isomorphic
  to~$C_0[0,\omega_1)$.
\end{corollary}

Combining Theorem~\ref{thmcharloywillis} with the techniques developed
by Willis in~\cite{willis}, we obtain a very short proof of Ogden's
main theorem~\cite[Theorem~6.18]{ogden} as it applies to the
ordinal~$\omega_1$. 

\begin{corollary}[Ogden]\label{ogdenthm}
Each algebra homomorphism from~$\mathscr{B}(C_0[0,\omega_1))$ into a
  Banach algebra is automatically continuous.
\end{corollary}

Our final result relies on a suitable modification of work of the
third-named author~\cite{laustsen}.

\begin{definition}
  Let~$\mathscr{A}$ be an algebra. The \emph{commutator} of a pair of
  elements~$a,b\in\mathscr{A}$ is given by $[a,b] = ab - ba$.  A
  \emph{trace} on~$\mathscr{A}$ is a scalar-valued, linear
  mapping~$\tau$ defined on~$\mathscr{A}$ such that $\tau(ab) =
  \tau(ba)$ for each pair $a,b\in\mathscr{A}$.
\end{definition}

\begin{corollary}\label{commsandtraces}
Each operator belonging to the Loy--Willis ideal is the sum of at most
three commutators. 

Hence a scalar-valued, linear mapping~$\tau$ defined
on~$\mathscr{B}(C_0[0,\omega_1))$ is a trace if and only if $\tau$ is
a scalar multiple of the Alspach--Benyamini character.  In particular,
each trace on~$\mathscr{B}(C_0[0,\omega_1))$ is automatically
continuous.
\end{corollary}

\begin{remark}\label{KtheoryBComega1} Building on Corollary~\ref{commsandtraces}, one can 
  prove that the $K_0$-group of the Banach
  algebra~$\mathscr{B}(C_0[0,\omega_1))$ is isomorphic to~$\mathbb{Z}$
  by arguments similar to those given in
  \cite[Section~4]{LaustsenKtheory}, while the $K_1$-group
  of~$\mathscr{B}(C_0[0,\omega_1))$ vanishes. A full proof of these
  results will be published elsewhere~\cite{KKL}.
\end{remark}

\section{Preliminaries}\label{sectPrelim}
\subsection*{General conventions}
Our notation and terminology are fairly standard. We shall now outline
the most important parts.  Let $X$ be a Banach space, always supposed
to be over the scalar field~$\mathbb{K}$, where $\mathbb{K} =
\mathbb{R}$ or $\mathbb{K} = \mathbb{C}$. We write~$B_X$ for the
closed unit ball of~$X$.  The dual space of~$X$ is~$X^*$, and
$\langle\,\cdot\,,\,\cdot\,\rangle$ denotes the duality bracket
between~$X$ and~$X^*$.

By an \emph{operator}, we understand a bounded, linear mapping between
Banach spaces. We write $\mathscr{B}(X)$ for the Banach algebra of all
operators on~$X$, and $\mathscr{B}(X,Y)$ for the Banach space of all
operators from~$X$ to some other Banach space~$Y$. For an
operator~$T\in\mathscr{B}(X,Y)$, we denote by $T^*\in
\mathscr{B}(Y^*,X^*)$ its adjoint, while~$I_X$ is the identity
operator on~$X$.

Given Banach spaces~$W$, $X$, $Y$ and~$Z$ and operators~$S\colon W\to
X$ and~$T\colon Y\to Z$, we say that~$S$ \emph{factors through}~$T$ if
$S = UTR$ for some operators $R\colon W\to Y$ and $U\colon Z\to
X$. The following elementary characterization of the operators that
the identity operator factors through is well known.
\begin{lemma}\label{operatorsfactoringID}
  Let $X$, $Y$ and~$Z$ be Banach spaces, and let $T\colon X\to Y$ be
  an operator. Then the identity operator on~$Z$ factors through~$T$
  if and only if $X$ contains a closed subspace~$W$ such that:
  \begin{itemize}
  \item $W$ is isomorphic to~$Z;$
  \item the restriction of~$T$ to~$W$ is bounded below, in the sense
    that there exists a constant $\epsilon>0$ such that
    $\|Tw\|\ge\epsilon\|w\|$ for each $w\in W;$
  \item the image of~$W$ under~$T$ is complemented in~$Y$.
  \end{itemize}
\end{lemma}

For a Hausdorff space~$K$, $C(K)$ denotes the vector space of
scalar-valued, continuous functions on~$K$. In the case where $K$ is
locally compact, we write $C_0(K)$ for the subspace con\-sisting of
those functions~$f\in C(K)$ which `vanish at infinity', in the sense
that the set \mbox{$\{x\in K : |f(x)|\ge\epsilon\}$} is compact for
each $\epsilon >0$. Then $C_0(K)$ is a Banach space with respect to
the supremum norm. Alternatively, one may define $C_0(K)$ as
\[ C_0(K) = \{ f\in C(\widetilde{K}) : f(\infty) = 0\}, \] where
$\widetilde{K} = K\cup\{\infty\}$ is the one-point compactification
of~$K$.  We identify the dual space of~$C_0(K)$ with the Banach space
of scalar-valued, regular Borel measures on~$K$, and we shall
therefore freely use measure-theoretic terminology and notation when
dealing with functionals on~$C_0(K)$.  Given $x\in K$, we denote by
$\delta_x$ the Dirac measure at~$x$.

Lower-case Greek letters such as $\alpha$, $\beta$, $\gamma$, $\xi$,
$\eta$ and~$\zeta$ denote ordinals. The first infinite ordinal
is~$\omega$, while the first uncountable ordinal is~$\omega_1$. By
convention, we consider~$0$ a limit ordinal.  We use standard interval
notation for intervals of ordinals, so that, given a pair of ordinals
$\alpha\le\beta$, we write $[\alpha, \beta]$ and $[\alpha, \beta)$ for
the sets of ordinals~$\gamma$ such that $\alpha\le\gamma\le\beta$ and
$\alpha\le\gamma <\beta$, respectively.

For a non-zero ordinal $\alpha$, we equip the ordinal
interval~$[0,\alpha)$ with the order topology, which turns it into a
locally compact Hausdorff space that is compact if and only
if~$\alpha$ is a successor ordinal.  (According to the standard
construction of the ordinals, the interval~$[0,\alpha)$ is of course
equal to the ordinal~$\alpha$; we use the notation~$[0,\alpha)$ to
emphasize its structure as a topological space.)  Since~$[0,\alpha)$
is scattered, a classical result of Rudin~\cite{Ru} states that each
regular Borel measure on~$[0,\alpha)$ is purely atomic, so that the
dual space of~$C_0[0,\alpha)$ is isometrically isomorphic to the
Banach space \[ \ell_1(\alpha) = \biggl\{ g\colon
[0,\alpha)\to\mathbb{K} : \sum_{\beta<\alpha} \bigl|g(\beta)\bigr|
<\infty\biggr\} \] via the mapping
\begin{equation}\label{dualofC0isl1} 
  g\mapsto\sum_{\beta<\alpha} g(\beta)\delta_\beta,\quad
  \ell_1(\alpha)\to C_0[0,\alpha)^*. \end{equation}
This implies in particular that each operator~$T$ on~$C_0[0,\alpha)$
can be represented by a scalar-valued \mbox{$[0,\alpha)\times
  [0,\alpha)$}-matrix $(T_{\beta,\gamma})_{\beta,\gamma < \alpha}$
with absolutely summable rows.  The $\beta^{\text{th}}$ row of this
matrix is simply the Rudin representation of the
functional~$T^*\delta_\beta$; that is, $(T_{\beta,\gamma})_{\gamma
  <\alpha}$ is the uniquely determined element of~$\ell_1(\alpha)$
such that
\begin{equation}\label{Rudinmatrix} T^*\delta_\beta = \sum_{\gamma <
  \alpha}T_{\beta,\gamma}\delta_\gamma. \end{equation} This matrix
representation plays an essential role in the original definition of
the Loy--Willis ideal, which is our next topic.

\subsection*{The Loy--Willis ideal}
Suppose that $\alpha = \omega_1+1$ in the notation of the previous
paragraph, and note that $C_0[0,\omega_1+1) = C[0,\omega_1]$. Using
the fact that each scalar-valued, continuous function
on~$[0,\omega_1]$ is eventually constant, Loy and Willis
\cite[Proposition~3.1]{lw} proved that, for each operator~$T$
on~$C[0,\omega_1]$, the $\gamma^{\text{th}}$ column of its matrix,
$k_\gamma^T\colon \beta\mapsto T_{\beta,\gamma}$, considered as a
scalar-valued func\-tion on~$[0,\omega_1]$, has the following three
continuity properties:
\begin{romanenumerate}
\item\label{LoyWillis3.1ii} $k^T_\gamma$ is continuous whenever
  $\gamma=0$ or $\gamma$ is a countable successor ordinal;
\item\label{LoyWillis3.1iii} $k^T_\gamma$ is continuous at $\omega_1$
  for each countable ordinal~$\gamma;$
\item\label{LoyWillis3.1iv} the restriction of $k^T_{\omega_1}$ to
  $[0,\omega_1)$ is continuous, and $\lim_{\beta\to
    \omega_1}k^T_{\omega_1}(\beta)$ exists.
\end{romanenumerate}
Clause~\romanref{LoyWillis3.1iv} is the best possible because the
final column of the matrix associated with the identity operator is
equal to the indicator function~$\mathbf{1}_{\{\omega_1\}}$, which is
dis\-con\-tinuous at~$\omega_1$. Hence, as Loy and Willis observed,
the set
\begin{equation}\label{defnLoyWillisIdeal}
  \mathscr{M} = \{ T\in \mathscr{B}(C[0,\omega_1]):
  k^T_{\omega_1}\ \text{is continuous
    at}\ \omega_1\} \end{equation} is a linear subspace of codimension
one in~$\mathscr{B}(C[0,\omega_1])$.  Since the composition of
operators on~$C[0,\omega_1]$ corresponds to matrix multiplication, in
the sense that
\begin{equation*}
  (ST)_{\alpha,\gamma}=\sum_{\beta\le\omega_1}S_{\alpha,
    \beta}T_{\beta,\gamma}\qquad
  (S,T\in\mathscr{B}(C[0,\omega_1]),\, \alpha,\gamma\in
  [0,\omega_1]),
\end{equation*}
$\mathscr{M}$ is a left ideal, named the \emph{Loy--Willis ideal}
in~\cite{kanialaustsen}. Having codimesion one, $\mathscr{M}$ is
automatically a maximal and two-sided ideal
of~$\mathscr{B}(C[0,\omega_1])$.

Consequently, the Banach algebra~$\mathscr{B}(C_0[0,\omega_1))$ also
contains a maximal ideal of codimension one because it is isomorphic
to~$\mathscr{B}(C[0,\omega_1])$. Loy and Willis's
definition~\eqref{defnLoyWillisIdeal} does not carry over
to~$\mathscr{B}(C_0[0,\omega_1))$ because the matrix of an operator
on~$C_0[0,\omega_1)$ has no final column. Instead we shall define the
Loy--Willis ideal of~$\mathscr{B}(C_0[0,\omega_1))$ as follows. Choose
an isomorphism~$U$ of~$C[0,\omega_1]$ onto~$C_0[0,\omega_1)$, and
declare that an operator~$T$ on~$C_0[0,\omega_1)$ belongs to the
Loy--Willis ideal of~$\mathscr{B}(C_0[0,\omega_1))$ if and only if the
operator~$U^{-1}TU$ on~$C[0,\omega_1]$ belongs the original
Loy--Willis ideal~\eqref{defnLoyWillisIdeal}. Since the latter is a
two-sided ideal, this definition is independent of the choice of the
isomorphism~$U$. We shall denote by~$\mathscr{M}$ the Loy--Willis
ideal of~$\mathscr{B}(C_0[0,\omega_1))$ defined in this way; this
should not cause any confusion with the original Loy--Willis ideal
given by~\eqref{defnLoyWillisIdeal}.

\subsection*{Uniform Eberlein compactness.} The following theorem,
which combines work of Ben\-ya\-mi\-ni, Rudin and Wage~\cite{BRW} and
Ben\-ya\-mi\-ni and Starbird~\cite{BS}, collects several important
characterizations of uniform Eberlein compactness.
\begin{theorem}[Benyamini--Rudin--Wage and Benyamini--Starbird]%
  \label{bsbrwThm} The following four conditions are equivalent for a
  compact Hausdorff space~$K\colon$
  \begin{alphenumerate}
  \item\label{bsbrwThm1} $K$ is uniformly Eberlein compact;
  \item\label{bsbrwThm4} the Banach space~$C(K)$ is Hilbert-generated;
  \item\label{bsbrwThm5} the unit ball of~$C(K)^*$ is uniformly
    Eberlein compact in the weak$^*$ topology;
  \item\label{bsbrwThm3} there exists a family
    $\mathscr{F}=\bigcup_{n\in \N}\mathscr{F}_n$ of open
    $F_\sigma$-subsets of~$K$ such that:
    \begin{enumerate}
    \item\label{lemmaSimpleTopDich2b} whenever $x, y\in K$ are
      distinct, some $G\in \mathscr{F}$ separates $x$ and~$y$, in the
      sense that either $(x\in G$ and $y\notin G)$ or $(y\in G$ and
      $x\notin G);$ and
    \item\label{lemmaSimpleTopDich2c} $\sup_{x\in K}\bigl|\{
      G\in\mathscr{F}_n : x\in G\}\bigr|$ is finite for each $n\in\N$.
    \end{enumerate}
\end{alphenumerate}
\end{theorem}

Another important theorem that we shall require is the following
internal characterization of the Banach spaces which embed in a
Hilbert-generated Banach space.  It is closely related to the
equivalence of conditions~\alphref{bsbrwThm4} and~\alphref{bsbrwThm5}
above. We refer to \cite[Theorem~6.30]{HMVZ} for a proof.

\begin{theorem}\label{charUECandHG}
  A Banach space~$X$ embeds in a Hilbert-generated Banach space if and
  only if the unit ball of~$X^*$ is uniformly Eberlein compact in the
  weak$^*$ topology.
\end{theorem}

\subsection*{The ideal of Hilbert-generated operators}
The first-named author and Kochanek~\cite{kaniakochanek} have recently
introduced the notion of a \emph{weakly compactly generated operator}
as an operator whose range is contained in a weakly compactly
generated subspace of its codomain, and have shown that the collection
of all such operators forms a closed operator ideal in the sense of
Pietsch. We shall now define the analogous operator ideal
corresponding to the class of Hilbert-generated Banach spaces.

\begin{definition}
An operator $T$ between Banach spaces $X$ and $Y$ is
\emph{Hilbert-generated} if its range $T[X]$ is contained in a
Hilbert-generated subspace of~$Y$; that is, there exist a Hilbert
space~$H$ and an operator~$R\colon H\to Y$ such that
$T[X]\subseteq\overline{R[H]}$. We write $\mathscr{HG}(X,Y)$ for the
set of Hilbert-generated operators from~$X$ to~$Y$.
\end{definition}

\begin{proposition}\label{hilbertideal} 
\begin{romanenumerate}
\item\label{hilbertideal1} The class $\mathscr{HG}$ is a closed
  operator ideal.
\item\label{hilbertideal2} Let $X$ be a Banach space. Then the ideal
  $\mathscr{HG}(X)$ is proper if and only if $X$ is not
  Hilbert-generated.
\end{romanenumerate}
\end{proposition}

\begin{proof} 
\romanref{hilbertideal1}.  Every finite-rank operator is clearly
Hilbert-generated.

Let $W$, $X$, $Y$ and $Z$ be Banach spaces.  To see that
$\mathscr{HG}(X,Y)$ is closed under addition, suppose that
$T_1,T_2\in\mathscr{HG}(X,Y)$. For $n=1,2$, take a Hilbert space~$H_n$ and
an operator $R_n\colon H_n\to Y$ such that
$T_n[X]\subseteq\overline{R_n[H_n]}$, and define 
\[ R\colon\ (x_1,x_2)\mapsto R_1x_1 + R_2x_2,\quad H_1\oplus H_2\to
Y. \] This is clearly a bounded operator with respect to the
$\ell_2$-norm on~$H_1\oplus H_2$, and we have
\enlargethispage*{10pt}
\[ \overline{R[H_1\oplus H_2]} = 
\overline{R_1[H_1]+R_2[H_2]}\supseteq T_1[X] + T_2[X] =
(T_1+T_2)[X], \] which proves that $T_1+T_2\in\mathscr{HG}(X,Y)$.
\newpage

Next, given $S\in\mathscr{B}(W,X)$, $T\in\mathscr{HG}(X,Y)$ and
$U\in\mathscr{B}(Y,Z)$, take a Hilbert space~$H$ and an operator
$R\colon H\to Y$ such that $T[X]\subseteq\overline{R[H]}$. Then, by
the continuity of~$U$, we obtain $UTS[W]\subseteq\overline{UR[H]}$, so
that $UTS\in\mathscr{HG}(W,Z)$.

Finally, suppose that $(T_n)_{n\in\N}$ is a norm-convergent sequence
in~$\mathscr{HG}(X,Y)$ with limit~$T$, say. For each $n\in\N$, take a
Hilbert space~$H_n$ and a contractive operator $R_n\colon H_n\to Y$ such that
$T_n[X]\subseteq\overline{R_n[H_n]}$. 
The Cauchy--Schwarz inequality ensures that
\[ R\colon\ (x_n)_{n\in\N}\mapsto \sum_{n\in\N}\frac{1}{n}R_nx_n \]
defines a bounded operator from the Hilbert space $H =
\bigl(\bigoplus_{n\in\N} H_n\bigr)_{\ell_2}$ into~$Y$.  The fact that
$R_n[H_n]\subseteq R[H]$ implies that $T_n[X]\subseteq\overline{R[H]}$
for each $n\in\N$, and therefore $T[X]\subseteq\overline{R[H]}$, which
proves that $T\in\mathscr{HG}(X,Y)$.

\romanref{hilbertideal2}. This is immediate because the range of the
identity operator on~$X$ is contained in a Hilbert-generated subspace
of~$X$ if and only if $X$ itself is Hilbert-generated.
\end{proof}

\mathversion{bold}
\subsection*{The Banach space $C_0(L_0)$} 
\mathversion{normal} The $c_0$-direct sum of a family $(X_j)_{j\in J}$
of Banach spaces is given by
\[ \biggl(\bigoplus_{j\in J}X_j\biggr)_{c_0} = \bigl\{ (x_j)_{j\in J}
: x_j\in X_j\ (j\in J)\ \text{and}\ \{ j\in J : \|x_j\|\ge\epsilon\}\
\text{is finite for each}\ \epsilon>0\bigr\}. \] In the case where
$X_j = X$ for each $j\in J$, we write $c_0(J,X)$ instead of
$\bigl(\bigoplus_{j\in J}X_j\bigr)_{c_0}$.  This no\-tion is relevant
for our purposes due to the following well-known elementary lemma
(\emph{e.g.}, see \cite[p.~191, Exercise 9]{conway}).

\begin{lemma}\label{conwaylemma}
  Let $L$ be the disjoint union of a family $(L_j)_{j\in J}$ of
  locally compact Hausdorff spaces. Then $C_0(L)$ is isometrically
  isomorphic to $\bigl(\bigoplus_{j\in J}C_0(L_j)\bigr)_{c_0}$.
\end{lemma}

Since the locally compact space~$L_0$ given by~\eqref{DefnL00} is the
disjoint union of the compact ordinal intervals $[0,\alpha]$ for
$\alpha<\omega_1$, this implies in particular that
\begin{equation}\label{C0L0asc0dirsum}
  C_0(L_0)\cong \biggl(\bigoplus_{\alpha<\omega_1}C[0,\alpha]\biggr)_{c_0}.
\end{equation}

\begin{corollary}\label{c0sumL0} The Banach space $C_0(L_0)$ is
  isomorphic to the $c_0$-direct sum of countably many copies of
  itself.
\end{corollary}
\begin{proof}
  It is well known that~$C[0,\alpha]$ is isomorphic to
  $c_0(\N,C[0,\alpha])$ for each $\alpha\in[\omega,\omega_1)$, with
  the Banach--Mazur distance bounded uniformly in~$\alpha$
  (\emph{e.g.}, see~\cite[Theorem~2.24]{rosen}). Hence,
  by~\eqref{C0L0asc0dirsum}, we have
  \[ c_0(\N, C_0(L_0))\cong
  \biggl(\bigoplus_{\alpha<\omega_1}c_0(\N,C[0,\alpha])\biggr)_{c_0}\cong
  \biggl(\bigoplus_{\alpha<\omega_1}C[0,\alpha]\biggr)_{c_0}\cong
  C_0(L_0), \] as desired.
\end{proof}

\begin{lemma}\label{c0sumHilbertgenerated}
  Let $(X_j)_{j\in J}$ be a family of Hilbert-generated Banach
  spaces. Then the Banach space $\bigl(\bigoplus_{j\in J}
  X_j\bigr)_{c_0}$ is Hilbert-generated.
\end{lemma}
\begin{proof}
  For each $j\in J$, choose a Hilbert space $H_j$ and a contractive
  operator $T_j\colon H_j\to X_j$ with dense range. The formula
  $(x_j)_{j\in J}\mapsto (T_jx_j)_{j\in J}$ then defines a contractive
  operator from the Hilbert space $\bigl(\bigoplus_{j\in J}
  H_j\bigr)_{\ell_2}$ onto a dense subspace of~$\bigl(\bigoplus_{j\in
    J} X_j\bigr)_{c_0}$.
\end{proof}

\begin{corollary}\label{L0uniformE} The Banach space $C_0(L_0)$ is 
  Hilbert-generated, and the one-point compactification of~${L}_0$ is
  therefore uniformly Eberlein compact.
\end{corollary}
\begin{proof} The first statement follows immediately
  from~\eqref{C0L0asc0dirsum} and Lemma~\ref{c0sumHilbertgenerated}
  because $C[0,\alpha]$ is separable and thus Hilbert-generated for
  each countable ordinal~$\alpha$.  Theorem~\ref{bsbrwThm} then
  implies the second part.
\end{proof}

By contrast, we have the following well-known result
for~$C_0[0,\omega_1)$.

\begin{theorem}\label{w1notUEC} The  Banach space~$C_0[0,\omega_1)$ does 
  not embed in any weakly compact\-ly gene\-rated Banach space.
\end{theorem}
\begin{proof} Every weakly compact\-ly gene\-rated Banach space is
  weakly Lindel\"{o}f (\emph{e.g.}, see
  \cite[Theorem~12.35]{fabianetal}), and this property is inherited by
  closed subspaces. However, $C_0[0,\omega_1)$ is not weakly
  Lindel\"{o}f.
\end{proof}

Combining this result with Amir and Lindenstrauss's theorem that a
compact space $K$ is Eberlein compact if and only if $C(K)$ is weakly
compact\-ly gene\-rated, we obtain the following conclusion, which can
also be proved directly (\emph{e.g.}, see
\cite[Exercises~12.58--59]{fabianetal}).

\begin{corollary}\label{omega1notEC}
The ordinal interval~$[0,\omega_1]$ is not Eberlein compact.
\end{corollary}

\subsection*{Club subsets}
A subset~$D$ of~$[0,\omega_1)$ such that $D$ is closed and unbounded
is a \emph{club subset}. Hence~$D$ is a club subset if and only if $D$
is uncountable and $D\cup\{\omega_1\}$ is closed in~$[0,\omega_1]$.
The collection
\[ \mathscr{D} = \bigl\{ D\subseteq[0,\omega_1) : D\ \text{contains a
  club subset of}\ [0,\omega_1)\bigr\}\] is a filter on the
set~$[0,\omega_1)$, and $\mathscr{D}$ is countably complete, in the
sense that $\bigcap\mathscr{C}$ belongs to~$\mathscr{D}$ for each
countable subset~$\mathscr{C}$ of~$\mathscr{D}$.

The following lemma is a variant of
\cite[Lemma~1.1(c)--(d)]{alspachbenyamini}, tailored to suit our
applications. Its proof is fairly straightforward, so we omit the
details.

\begin{lemma}\label{phiDlemma07082012}
  Let $D$ be a club subset of~$[0,\omega_1)$.
  \begin{romanenumerate}
  \item\label{phiDlemma07082012iii} The order isomorphism
    $\psi_D\colon[0,\omega_1)\to D$ is a homeomorphism, and hence the
    composition operator $U_D\colon g\mapsto g\circ\psi_D$ is an
    isometric isomorphism of~$C_0(D)$ onto~$C_0[0,\omega_1)$.
  \item\label{phiDlemma07082012ii} The mapping
  \begin{equation}\label{phiDlemma07082012iiEq1} \pi_D\colon\
    \alpha\mapsto\min\bigl(D\cap[\alpha,\omega_1)\bigr),\quad
    [0,\omega_1)\to D, \end{equation} 
  is an increasing retraction, and hence the
  composition operator $S_D\colon g\mapsto g\circ\pi_D$ is a
  linear isometry of~$C_0(D)$ into~$C_0[0,\omega_1)$.
\item\label{phiDlemma07082012i} Let $\iota_D\colon D\to[0,\omega_1)$
  denote the inclusion mapping. Then the composition operator
  $R_D\colon f\mapsto f\circ\iota_D$ is a linear contraction
  of~$C_0[0,\omega_1)$ into~$C_0(D)$, and $R_DS_D = I_{C_0(D)}$.
 \item\label{phiDlemma07082012iv} The operator $P_D = S_DR_D$ is
   a contractive projection on~$C_0[0,\omega_1)$ such that
  \begin{equation}\label{phiDlemma07082012eq2}
    \ker P_D = \bigl\{f\in C_0[0,\omega_1): f(\alpha) =0\ (\alpha\in
      D)\bigr\},
  \end{equation} and the
  range of~$P_D$,
  \begin{equation}\label{phiDlemma07082012eq3} 
    \mathscr{R}_D= \bigl\{
    f|_D\circ \pi_D: f\in C_0[0,\omega_1)\bigr\},
   \end{equation} is isometrically isomorphic
   to~$C_0[0,\omega_1)$.
 \item\label{phiDlemma07082012v} Suppose that
   $D\neq[0,\omega_1)$. Then $[0,\omega_1)\setminus D =
   \bigcup_{\alpha<\gamma} [\xi_\alpha,\eta_\alpha)$ for some ordinal
   \mbox{$\gamma\in[1,\omega_1]$} and some sequences
   $(\xi_\alpha)_{\alpha< \gamma}$ and $(\eta_\alpha)_{\alpha<
     \gamma}$, where $\xi_\alpha$ is either~$0$ or a countable
   successor ordinal, $\eta_\alpha\in D$ and $\xi_\alpha < \eta_\alpha
   <\xi_{\alpha+1}$ for each~$\alpha$, and $\ker P_D$ is isometrically
   isomorphic to~$\bigl(\bigoplus_{\alpha<\gamma}
   C_0[\xi_\alpha,\eta_\alpha)\bigr)_{c_0}$.
\end{romanenumerate}
\end{lemma}

\begin{corollary}\label{OmegaXiLemma10Jan2013ii} 
  For each club subset~$D$ of~$[0,\omega_1)$, $\ker P_D$ is
  Hilbert-generated and iso\-metrically isomorphic to a complemented
  subspace of $C_0(L_0)$.

  Hence the operator $I_{C_0[0,\omega_1)} - P_D$ belongs to the
  ideal~$\mathscr{HG}(C_0[0,\omega_1))$.
\end{corollary}
\begin{proof}
  The result is trivial if $D = [0,\omega_1)$ because $\ker P_D =
  \{0\}$ in this case. Otherwise
  Lemma~\ref{phiDlemma07082012}\romanref{phiDlemma07082012v} applies,
  and the conclusions follow using Lemma~\ref{c0sumHilbertgenerated}
  and~\eqref{C0L0asc0dirsum}.
\end{proof}

\begin{corollary}\label{OmegaXiLemma10Jan2013iii} 
  There exists a club subset~$D$ of~$[0,\omega_1)$ such that $\ker
  P_D$ is isometrically isomorphic to~$C_0(L_0)$.
\end{corollary}
\begin{proof}
  We can inductively define a transfinite sequence
  $(\xi_\alpha)_{\alpha<\omega_1}$ of countable ordinals by $\xi_0 =
  0$ and $\xi_\alpha = \sup_{\beta<\alpha}(\xi_\beta+\beta)+2$ for
  each $\alpha\in[0,\omega_1)$.  Then $D =
  [0,\omega_1)\setminus\bigcup_{\alpha<\omega_1}[\xi_\alpha,\xi_\alpha+\alpha]$
  is a proper club subset of~$[0,\omega_1)$.  In the notation of
  Lemma~\ref{phiDlemma07082012}\romanref{phiDlemma07082012v}, we have
  $\gamma = \omega_1$ and $\eta_\alpha = \xi_\alpha+\alpha+1$ for each
  $\alpha<\omega_1$, and hence
  \[ \ker P_D\cong \biggl(\bigoplus_{\alpha<\omega_1}
  C_0[\xi_\alpha,\eta_\alpha)\biggr)_{c_0}\cong
  \biggl(\bigoplus_{\alpha<\omega_1} C[0,\alpha]\biggr)_{c_0}\cong
  C_0(L_0) \] by~\eqref{C0L0asc0dirsum}, as desired.
\end{proof}

\begin{lemma}\label{intersections} Let $D$ and $E$ be club subsets
  of~$[0,\omega_1)$. Then
  \[ \ker P_D\cap \ker P_E = \ker P_{D\cup E}\qquad
  \text{and}\qquad \mathscr{R}_{D}\cap
  \mathscr{R}_{E}=\mathscr{R}_{D\cap E}. \]
\end{lemma}
\begin{proof} 
  The first identity is an immediate consequence
  of~\eqref{phiDlemma07082012eq2}.

  To verify the second, suppose first that $f\in \mathscr{R}_{D}\cap
  \mathscr{R}_{E}$. Given $\alpha\in[0,\omega_1)$, an easy
  trans\-finite induction shows that $f(\beta) = f(\alpha)$ for each
  $\beta\in [\alpha,\pi_{D\cap E}(\alpha)]$, so that in particular we
  have $f(\alpha) = f(\pi_{D\cap E}(\alpha))$, and hence
  $f\in\mathscr{R}_{D\cap E}$.

  Conversely, for each $\alpha\in[0,\omega_1)$, we see that
  $D\cap[\pi_D(\alpha),\omega_1) = D\cap
  [\alpha,\omega_1)$. Consequently $D\cap
  E\cap[\pi_D(\alpha),\omega_1) = D\cap E\cap[\alpha,\omega_1)$, so
  that $\pi_{D\cap E}(\pi_D(\alpha)) = \pi_{D\cap E}(\alpha)$, and therefore
  \[ (P_DP_{D\cap E}f)(\alpha) = f(\pi_{D\cap E}(\pi_D(\alpha))) =
  f(\pi_{D\cap E}(\alpha)) = (P_{D\cap E}f)(\alpha)\qquad (f\in
  C_0[0,\omega_1)). \] This proves that $\mathscr{R}_{D\cap
    E}\subseteq\mathscr{R}_D$. A similar argument shows that
  $\mathscr{R}_{D\cap E}\subseteq\mathscr{R}_E$.
\end{proof}

\section{The proof of Theorem~\ref{dichotomy}}\label{section3}
\noindent
\begin{lemma}\label{copiesofw1} Let  $\lambda\in\mathbb{K}\setminus\{0\}$
  and $\rho\in C_0[0,\omega_1)^*$. The mapping
  \mbox{$\sigma_{\lambda,\rho}\colon [0,\omega_1]\to C_0[0,\omega_1)^*$}
  given by $\sigma_{\lambda,\rho}(\alpha) = \lambda\delta_\alpha+\rho$
  for $\alpha<\omega_1$ and $\sigma_{\lambda,\rho}(\omega_1)=\rho$ is then
  injective and continuous with respect to the weak$^*$ topology on
  its codomain.

  Hence its range, which is equal to $\{\lambda\delta_\alpha+\rho:
  \alpha<\omega_1\}\cup\{\rho\}$, is homeomorphic to $[0,\omega_1]$.
\end{lemma}

\begin{proof} It is well known that the mapping
  $\tau\colon[0,\omega_1]\to C_0[0,\omega_1)^*$ given by $\tau(\alpha)
  = \delta_\alpha$ for $\alpha\in [0,\omega_1)$ and $\tau(\omega_1) =
  0$ is a continuous injection with respect to the weak$^*$ topology
  on its codomain, and hence the same is true for
  $\sigma_{\lambda,\rho}$ because $\lambda\neq 0$ and the vector-space
  operations are weakly$^*$ continuous.  The final clause now follows
  because~$[0,\omega_1]$ is compact, and the weak$^*$ topology
  on~$C_0[0,\omega_1)^*$ is Hausdorff.
\end{proof}

\begin{definition}
  A subset~$S$ of~$[0,\omega_1)$ is \emph{stationary} if $S\cap
  D\neq\emptyset$ for each club subset~$D$ of~$[0,\omega_1)$.
\end{definition}

Stationary sets have many interesting topological and combinatorial
properties, as indicated in~\cite{jech} and~\cite{kunen}, for
instance. We shall only require the following result, which is due to
Fodor~\cite{fodor}.

\begin{theorem}[Pressing Down Lemma]\label{fodorsthm}
  Let $S$ be a stationary subset of~$[0,\omega_1)$, and let \mbox{$f\colon
  S\to[0,\omega_1)$} be a function which satisfies $f(\alpha)<\alpha$
  for each $\alpha\in S$.  Then~$S$ contains a subset~$S'$ which is
  stationary and for which $f|_{S'}$ is constant.
\end{theorem}

We can now explain how the proof of Theorem~\ref{dichotomy} is
structured: it consists of three parts, set out in the following
lemma.  Theorem~\ref{dichotomy} follows immediately from it, using
Lemma~\ref{copiesofw1}.

\begin{lemma}\label{lemmaSimpleTopDich}
  Let $K$ be a weakly$^*$ compact subset~$K$ of~$C_0[0,\omega_1)^*$.
  \begin{romanenumerate}
  \item\label{lemmaSimpleTopDich1} Exactly one of the following two
    alternatives holds:
    \begin{capromanenumerate}
    \item\label{lemmaSimpleTopDich1a} either there is a club
      subset~$D$ of~$[0,\omega_1)$ such that
      \begin{equation}\label{lemmaSimpleTopDich2Eq1}
        \mu([\alpha,\omega_1)) = 0\qquad (\mu\in K,\,\alpha\in D);
      \end{equation}
    \item\label{lemmaSimpleTopDich1b} or the set
      \begin{equation}\label{lemmaSimpleTopDichEq1} \{
        \alpha\in[0,\omega_1)
        : \mu([\alpha,\omega_1))\neq 0\ \text{\normalfont{for some}}\
        \mu\in K \} \end{equation} 
      is stationary.
    \end{capromanenumerate}
  \item\label{lemmaSimpleTopDich2}
    Condition~\romanref{lemmaSimpleTopDich1a} above is satisfied if
    and only if $K$ is uniformly Eberlein compact.
  \item\label{lemmaSimpleTopDich3}
    Condition~\romanref{lemmaSimpleTopDich1b} above is satisfied if
    and only if there exist $\rho\in K$,
    $\lambda\in\mathbb{K}\setminus\{0\}$, and a club subset~$D$
    of~$[0,\omega_1)$ such that $\rho+\lambda\delta_\alpha\in K$ for
    each $\alpha\in D$.
  \end{romanenumerate}
\end{lemma}

In the proof, we shall require the following well-known, elementary
observations. 

\begin{lemma}\label{littlelemma16nov}
  Let $\mu\in C_0[0,\omega_1)^*$, and let $\alpha\in[\omega,\omega_1)$
  be a limit ordinal. Then, for each $\epsilon >0$, there exists an
  ordinal $\alpha_0<\alpha$ such that $|\mu([\beta,\alpha))|<\epsilon$
  whenever $\beta\in [\alpha_0,\alpha)$.
\end{lemma}

\begin{proof} Straightforward! \end{proof}

\begin{lemma}\label{lemma1Nov29}
  \begin{romanenumerate}
  \item\label{lemma1Nov29a} Let $\{ S_n : n\in\N\}$ be a countable
    family of subsets of~$[0,\omega_1)$ such that $\bigcup_{n\in\N}
    S_n$ is a stationary subset of~$[0,\omega_1)$. Then $S_n$ is
    stationary for some~$n\in\N$.
  \item\label{lemma1Nov29b} Let $S$ be a stationary subset
    of~$[0,\omega_1)$, and let~$D$ be a club subset
    of~$[0,\omega_1)$. Then $S\cap D$ is stationary.
  \end{romanenumerate}
\end{lemma}
\begin{proof} \romanref{lemma1Nov29a}. Suppose contrapositively that
  $S_n$ is not stationary for each~$n\in\N$, and take a club
  subset~$D_n$ of~$[0,\omega_1)$ such that $S_n\cap D_n =
  \emptyset$. Then $D = \bigcap_{n\in\N} D_n$ is a club subset
  of~$[0,\omega_1)$ such that $\bigl(\bigcup_{n\in\N} S_n\bigr)\cap D
  = \emptyset$, which shows that $\bigcup_{n\in\N} S_n$ is not
  stationary.

  \romanref{lemma1Nov29b}.  We have $(S\cap D)\cap E = S\cap (D\cap
  E)\neq\emptyset$ for each club subset~$E$ of~$[0,\omega_1)$ because
  $D\cap E$ is a club subset.
\end{proof}

\begin{lemma}\label{dualballwsseqcomp} Let $K$ be a  scattered locally
  compact space. Then the unit ball of $C_0(K)^*$ is weakly$^*$
  sequentially compact.

  In particular, the unit ball of $C_0[0,\omega_1)^*$ is weakly$^*$
  sequentially compact.
\end{lemma}
\begin{proof} The fact that $K$ is scattered implies that the Banach
  space~$C_0(K)$ is Asplund. Consequently, the unit ball of~$C_0(K)^*$
  in its weak$^*$ topology is Radon--Nikodym compact, and thus
  sequentially compact.
\end{proof}

\begin{proof}[Proof of Lemma~{\normalfont{\ref{lemmaSimpleTopDich}}}]
  Let~$S$ denote the set given by~\eqref{lemmaSimpleTopDichEq1}.
  Since weakly$^*$ compact sets are bounded, we may suppose that $K$
  is contained in the unit ball of~$C_0[0,\omega_1)^*$.

  Part~\romanref{lemmaSimpleTopDich1} is clear
  because~\romanref{lemmaSimpleTopDich1b} is the negation
  of~\romanref{lemmaSimpleTopDich1a}.

  \romanref{lemmaSimpleTopDich2}, $\Rightarrow$. Suppose that~$D$ is a
  club subset of~$[0,\omega_1)$ such
  that~\eqref{lemmaSimpleTopDich2Eq1} holds.  Replacing~$D$ with its
  intersection with the club subset of limit ordinals
  in~$[\omega,\omega_1)$, we may additionally suppose that~$D$
  consists entirely of infinite limit ordinals.  In the case of real
  scalars, let~$\Delta$ be the collection of open intervals
  $(q_1,q_2)$, where $q_1<q_2$ are rational and $0\notin(q_1,q_2)$.
  Otherwise $\mathbb{K} = \C$, in which case we define~$\Delta$ as the
  collection of open rectangles $(q_1,q_2)\times(r_1,r_2)$ in the
  complex plane, where $q_1<q_2$ and $r_1<r_2$ are rational and
  \mbox{$0 = (0,0)\notin(q_1,q_2)\times (r_1,r_2)$}. In both cases
  $\Delta$ is countable, so that we can take a bijection
  $\delta\colon\N\to\Delta$.

  For technical reasons, it is convenient to introduce a new limit
  ordinal, which is the predecessor of~$0$, and which we therefore
  suggestively denote by~$-1$.  Set $D' = D\cup\{-1\}$. For each
  $\alpha\in D'$, we define $\alpha^+ = \pi_D(\alpha+1) = \min
  (D\cap[\alpha+1,\omega_1))\in D$, using the notation of
  Lemma~\ref{phiDlemma07082012}\romanref{phiDlemma07082012ii}.
  Let~$\Gamma_\alpha$ denote the set of ordered pairs $(\xi,\eta)$ of
  ordinals such that \mbox{$\alpha\le\xi<\eta<\alpha^+$}, and take a
  bijection $\gamma_\alpha\colon\N\to\Gamma_\alpha$. Moreover, let
  \mbox{$\sigma\colon\N\to\N\times\N$} be a fixed bijection, chosen
  independently of~$\alpha$. We can then define a bijection by \[
  \tau_\alpha =
  (\gamma_\alpha\times\delta)\circ\sigma\colon\N\to\Gamma_\alpha\times\Delta\qquad
  (\alpha\in D'). \] Hence, for each $n\in\N$ and $\alpha\in D'$, we
  have $\tau_\alpha(n) = (\xi,\eta,R)$ for some
  $(\xi,\eta)\in\Gamma_\alpha$ and $R\in\Delta$, where $R$ depends
  only on~$n$, not on~$\alpha$. Using this notation, we define
  \begin{equation*}
    G_\alpha^n = \{\mu\in K : \mu([\xi+1,\eta])\in R\},
  \end{equation*}
  which is a relatively weakly$^*$ open $F_\sigma$-subset of~$K$
  because~$R$ is an open $F_\sigma$-subset of~$\mathbb{K}$ and the
  indicator function $\mathbf{1}_{[\xi+1,\eta]}$ is continuous. Let
  $\mathscr{F}_n = \{G_\alpha^n : \alpha\in D'\}$. We shall now
  complete the proof of \romanref{lemmaSimpleTopDich2}, $\Rightarrow$,
  by verifying that the family $\mathscr{F} =
  \bigcup_{n\in\N}\mathscr{F}_n$ satisfies
  condi\-tions~\eqref{lemmaSimpleTopDich2b}--\eqref{lemmaSimpleTopDich2c}
  of Theorem~\ref{bsbrwThm}\alphref{bsbrwThm3}.

  \eqref{lemmaSimpleTopDich2b}. Suppose that $\mu,\nu\in K$ are
  distinct. Since~$\mu$ and~$\nu$ are purely atomic, we have
  $\mu(\{\alpha\})\neq\nu(\{\alpha\})$ for some
  $\alpha\in[0,\omega_1)$. By interchanging $\mu$ and~$\nu$ if
  necessary, we may suppose that \mbox{$\mu(\{\alpha\})\neq 0$}, in
  which case there exists $R\in\Delta$ such that $\mu(\{\alpha\})\in
  R$ and $\nu(\{\alpha\})\notin\overline{R}$. We shall now split into
  two cases.

  Suppose first that $\alpha$ belongs to~$D$. Then, as $\alpha^+$ also
  belongs to~$D$, \eqref{lemmaSimpleTopDich2Eq1} implies that
  \[ \mu([\alpha+1,\alpha^+)) = \mu([\alpha, \omega_1))
  -\mu(\{\alpha\}) - \mu([\alpha^+,\omega_1)) = 0-\mu(\{\alpha\})-0\in
  -R, \] and similarly $\nu([\alpha+1,\alpha^+)) =
  -\nu(\{\alpha\})\notin -\overline{R}$. Since $-R$ and the complement
  of~$-\overline{R}$ are open, and $\alpha^+$ is a limit ordinal,
  Lemma~\ref{littlelemma16nov} enables us to find
  $\eta\in[\alpha+1,\alpha^+)$ such that
  \begin{equation}\label{lemmaSimpleTopDich2bEq1}
    \mu([\alpha+1,\eta))\in -R\qquad\text{and}\qquad
    \nu([\alpha+1,\eta))\notin-\overline{R}.
  \end{equation}
  The pair $(\alpha,\eta)$ then belongs to~$\Gamma_\alpha$,
  and~\eqref{lemmaSimpleTopDich2bEq1} shows that $\mu\in G_\alpha^n$
  and $\nu\notin G_\alpha^n$ for $n =
  \tau_\alpha^{-1}(\alpha,\eta,-R)\in\N$, as desired.

  Secondly, in the case where $\alpha\notin D$ we can take $\beta\in
  D'$ such that $\beta<\alpha<\beta^+$. (This is where the
  introduction of the new ordinal~$-1$ is useful.)  If $\alpha =
  \zeta+1$ for some ordinal~$\zeta$, then the pair $(\zeta,\alpha)$
  belongs to~$\Gamma_\beta$, so that we can define $n =
  \tau_\beta^{-1}(\zeta,\alpha,R)\in\N$, and we have $\mu\in
  G_\beta^n$ and $\nu\notin G_\beta^n$ because $\mu([\zeta+1,\alpha])
  = \mu(\{\alpha\})\in R$ and 
  $\nu([\zeta+1,\alpha]) = \nu(\{\alpha\})\notin \overline{R}$.

  Otherwise $\alpha$ is an infinite limit ordinal. By
  Lemma~\ref{littlelemma16nov}, we can find $\xi\in[\beta,\alpha)$
  such that $\mu([\xi+1,\alpha))$ and $\nu([\xi+1,\alpha))$ are as
  small as we like. In particular, since $\mu(\{\alpha\})$ and
  $\nu(\{\alpha\})$ belong to the open sets~$R$
  and~$\mathbb{K}\setminus\overline{R}$, respectively, we can choose
  $\xi\in[\beta,\alpha)$ such that $\mu([\xi+1,\alpha])\in R$ and
  $\nu([\xi+1,\alpha])\notin\overline{R}$.  Hence the pair
  $(\xi,\alpha)$ belongs to~$\Gamma_\beta$, and we have $\mu\in
  G_\beta^n$ and $\nu\notin G_\beta^n$ for $n =
  \tau_\beta^{-1}(\xi,\alpha,R)\in\N$.

  \eqref{lemmaSimpleTopDich2c}. Assume towards a contradiction that
  \mbox{$\sup_{\mu\in K}\bigl|\{ G\in\mathscr{F}_n : \mu\in
    G\}\bigr|$} is in\-finite for some $n\in\N$, and let $(n_1,n_2) =
  \sigma(n)\in\N^2$.  We shall focus on the case of complex scalars
  because it is slightly more complicated than the real case.  Set
  \mbox{$R = \delta(n_2) = (q_1,q_2)\times (r_1,r_2)\in\Delta$}.
  Since $(0,0)\notin R$, either $0\notin (q_1,q_2)$ or $0\notin
  (r_1,r_2)$. Suppose that we are in the first case, and choose
  $m\in\N$ such that $m\cdot\min\{|q_1|,|q_2|\} > 1$. By the
  assumption, we can find $\mu\in K$ and ordinals
  $\alpha_1<\alpha_2<\cdots<\alpha_m$ in~$D'$ such that
  $\mu\in\bigcap_{j=1}^m G_{\alpha_j}^n$. Letting $(\xi_j,\eta_j) =
  \gamma_{\alpha_j}(n_1)\in\Gamma_{\alpha_j}$ for each
  $j\in\{1,\ldots,m\}$, we have $\tau_{\alpha_j}(n) =
  (\xi_j,\eta_j,R)$, so that $\mu([\xi_j+1,\eta_j])\in R$ because
  $\mu\in G_{\alpha_j}^n$.  The intervals
  $[\xi_1+1,\eta_1],[\xi_2+1,\eta_2],\ldots,[\xi_m+1,\eta_m]$ are
  disjoint because
  \[ \alpha_1\le \xi_1<\eta_1< \alpha_1^+\le\alpha_2\le
  \xi_2<\eta_2<\alpha_2^+\le\cdots\le \alpha_m\le \xi_m<\eta_m<
  \alpha_m^+ , \] and hence we conclude that
  \begin{align*} 1 &\ge\|\mu\| \ge\biggl|\mu\biggl(\bigcup_{j=1}^m
    [\xi_j+1,\eta_j]\biggr)\biggr| = \biggl|\sum_{j=1}^m
    \mu([\xi_j+1,\eta_j])\biggr|\\ &\ge\biggl|\realpart \sum_{j=1}^m
    \mu([\xi_j+1,\eta_j])\biggr| = \biggl| \sum_{j=1}^m \realpart
    \mu([\xi_j+1,\eta_j])\biggr| \ge m\cdot\min\{|q_1|,|q_2|\}
    >1, \end{align*} which is clearly absurd. The case where $0\notin
  (r_1,r_2)$ is very similar: we simply replace $\min\{|q_1|,|q_2|\}$
  and the real part with $\min\{|r_1|,|r_2|\}$ and the imaginary part,
  respectively.

  The case where $\mathbb{K} = \R$ is also similar, but easier,
  because there is no need to pass to the real part in the above
  calculation.

  \romanref{lemmaSimpleTopDich3},~$\Leftarrow$, is an easy consequence
  of the previous implications. Suppose contrapositively that
  condition~\alphref{lemmaSimpleTopDich1b} is not satisfied. Then,
  by~\romanref{lemmaSimpleTopDich1},
  condition~\alphref{lemmaSimpleTopDich1a} holds, so that $K$ is
  uniformly Eberlein compact by what we have just proved.  Each
  weakly$^*$ closed subset of $K$ is therefore also uniformly Eberlein
  compact, and hence Corollary~\ref{omega1notEC} implies that no
  subset of~$K$ is homeomorphic to~$[0,\omega_1]$. The desired
  conclusion now follows from Lemma~\ref{copiesofw1}.

  \romanref{lemmaSimpleTopDich3}, $\Rightarrow$.  Suppose that the
  set~$S$ given by~\eqref{lemmaSimpleTopDichEq1} is stationary. Since
  \[ S = \bigcup_{n\in\N}\Bigl\{\alpha\in [0,\omega_1) :
  \bigl|\mu([\alpha,\omega_1))\bigr| > \frac{1}{n}\
  \text{\normalfont{for some}}\ \mu\in K\Bigr\}, \]
  Lemma~\ref{lemma1Nov29}\romanref{lemma1Nov29a} implies that the set
  \[ S_0 = \bigl\{\alpha\in [0,\omega_1) :
  \bigl|\mu([\alpha,\omega_1))\bigr| > \epsilon_0\
  \text{\normalfont{for some}}\ \mu\in K\bigr\} \] is stationary for
  some $\epsilon_0 >0$. Replacing $S_0$ with its intersection with the
  club subset of limit ordinals in $[\omega,\omega_1)$, we may in
  addition suppose that~$S_0$ consists entirely of infinite limit
  ordinals by Lemma~\ref{lemma1Nov29}\romanref{lemma1Nov29b}.  For
  each $\alpha\in S_0$, take $\mu_\alpha\in K$ such that
  $\bigl|\mu_\alpha([\alpha,\omega_1))\bigr| > \epsilon_0$.
  Lemma~\ref{littlelemma16nov} implies that
  $|\mu_\alpha|([f(\alpha),\alpha))<\epsilon_0/3$ for some ordinal
  $f(\alpha)\in[0,\alpha)$, where~$|\mu_\alpha|$ denotes the total
  varia\-tion of $\mu_\alpha$, that is, the positive measure
  on~$[0,\omega_1)$ given by \[ |\mu_\alpha|(B) = \sum_{\beta\in
    B}\bigl|\mu_\alpha(\{\beta\})\bigr|\qquad
  (B\subseteq[0,\omega_1)). \] By Theorem~\ref{fodorsthm}, $S_0$
  contains a subset~$S'$ which is stationary and for which $f|_{S'}$
  is constant, say $f(\alpha) = \zeta_0$ for each $\alpha\in S'$.

  Define $\mathbb{L} = \mathbb{Q}$ for $\mathbb{K} = \mathbb{R}$ and
  $\mathbb{L} = \{ q + r\mathrm{i} : q,r\in\mathbb{Q}\}$ for
  $\mathbb{K} = \mathbb{C}$, so that $\mathbb{L}$ is a countable,
  dense subfield of~$\mathbb{K}$.  For each $\alpha\in S'$ and
  $k\in\N$, choose a non-empty, finite subset~$F_{\alpha,k}$
  of~$[0,\omega_1)$ and scalars $q_{\alpha,k}^\beta\in\mathbb{L}$ for
  $\beta\in F_{\alpha,k}$ such that
  \begin{equation}\label{eqDefnmualphak}
    \mu_{\alpha,k} = \sum_{\beta\in
      F_{\alpha,k}}q_{\alpha,k}^\beta\delta_\beta\in
    C_0[0,\omega_1)^*
  \end{equation} has norm at most one and satisfies 
  \begin{equation}\label{eqDefnmualphak2}
    \| \mu_{\alpha,k} - \mu_\alpha\| <
    \min\Bigl\{\frac{\epsilon_0}{3},\frac{1}{k}\Bigr\}.
  \end{equation} 

  Suppose that $(\alpha_k)_{k\in \N}$ is a sequence in~$S'$ such that
  $(\mu_{\alpha_k, k})_{k\in\N}$ is weakly$^*$ convergent with
  limit~$\nu\in C_0[0,\omega_1)^*$, say. Then we claim that
  \begin{equation}\label{PKclaim1}
    \nu = \text{w}^*\text{-}\lim_k\mu_{\alpha_k},
  \end{equation}
  a conclusion which we shall require towards the end of the
  proof. Indeed, for each $\epsilon>0$ and $g\in C_0[0,\omega_1)$, we
  can choose $k_0\in\N$ such that $k_0> 2\| g\|/\epsilon$ and
  $\bigl|\langle g, \mu_{\alpha_k,k} - \nu\rangle\bigr| <\epsilon/2$
  whenever $k\ge k_0$, and hence
  \[ \bigl|\langle g, \mu_{\alpha_k} - \nu\rangle\bigr| \le
  \bigl|\langle g, \mu_{\alpha_k} - \mu_{\alpha_k,k}\rangle\bigr| +
  \bigl|\langle g, \mu_{\alpha_k,k} - \nu\rangle\bigr|<
  \frac{\|g\|}{k} + \frac{\epsilon}{2}<\epsilon\qquad (k\ge k_0). \]

  Fix $k\in\N$. Since $\{F_{\alpha,k} : \alpha\in S'\}$ is an
  uncountable collection of finite sets, the $\Delta$-system Lemma
  (see~\cite{shanin}, or \cite[Theorem~9.18]{jech} for an exposition)
  yields the existence of a set~$\Delta_k$ and an uncountable
  subset~$A_k$ of~$S'$ such that
  \begin{equation}\label{eqAk1}
    F_{\alpha,k}\cap F_{\beta,k} = \Delta_k\qquad (\alpha,\beta\in A_k,\,
    \alpha\neq\beta).
  \end{equation}
  We shall now arrange that a number of further properties hold by
  passing to suitably chosen uncountable subsets of~$A_k$.

  The fact that $A_k = \bigcup_{n\in\N}\{\alpha\in A_k :
  |F_{\alpha,k}| = n\}$ implies that, for some $n_k\in\N$,
  $A_k$~contains an uncountable subset~$A_k'$ such that
  $|F_{\alpha,k}| = n_k$ for each $\alpha\in A_k'$.  Let
  \mbox{$\theta_{\alpha,k}\colon\{1,\ldots,n_k\}\to F_{\alpha,k}$} be
  the unique order isomorphism for $\alpha\in A_k'$. Recall
  from~\eqref{eqDefnmualphak} that $q_{\alpha,k}^\beta\in\mathbb{L}$
  for $\beta\in F_{\alpha,k}$ are the coefficients
  of~$\mu_{\alpha,k}$. Since~$\mathbb{L}$ is countable and
  \[ A_k' = \bigcup_{q_1,\ldots,q_{n_k}\in\mathbb{L}} \bigl\{
  \alpha\in A_k': q_{\alpha,k}^{\theta_{\alpha,k}(j)} = q_j\ \text{for
    each}\ j\in\{1,\ldots,n_k\}\bigr\}, \] we can find $q_{1,k},\ldots,
  q_{n_k,k}\in\mathbb{L}$ and an uncountable subset~$A_k''$ of~$A_k'$
  such that
  \begin{equation}\label{eqAk5}
    q_{\alpha,k}^{\theta_{\alpha,k}(j)} = q_{j,k}\qquad
    (j\in\{1,\ldots,n_k\},\,\alpha\in A_k'').
  \end{equation}
  
  Our next aim is to show that~$A_k''$ contains an uncountable
  subset~$A_k'''$ such that
  \begin{equation}\label{eqAk3}
    \Delta_k\subsetneq F_{\alpha,k}\qquad (\alpha\in A_k''').
  \end{equation}
  This is trivially true if $\Delta_k$ is empty. Otherwise let
  \mbox{$A_k''' = A_k''\cap [\max\Delta_k+1,\omega_1)$}, which is
  uncountable, and assume towards a contradiction that $\Delta_k =
  F_{\alpha,k}$ for some $\alpha\in A_k'''$.
  By~\eqref{eqDefnmualphak}, we have
  $\mu_{\alpha,k}([\alpha,\omega_1)) = 0$, so that
  \[ \|\mu_\alpha - \mu_{\alpha,k}\|\ge \bigl|(\mu_\alpha -
  \mu_{\alpha,k})([\alpha,\omega_1))\bigr| =
  \bigl|\mu_\alpha([\alpha,\omega_1))\bigr| >\epsilon_0, \] which
  contradicts~\eqref{eqDefnmualphak2}. Hence~\eqref{eqAk3} is
  satisfied for the above choice of~$A_k'''$.

  For each $\beta\in[0,\omega_1)$, the set
  \[ B_k^\beta = \{\alpha\in A_k''' :
  \min(F_{\alpha,k}\setminus\Delta_k)\le\beta <\alpha \} =
  \bigcup_{\gamma\in [0,\beta]\setminus\Delta_k} \{\alpha\in
  A_k'''\cap[\beta+1,\omega_1): \gamma\in F_{\alpha,k}\}
  \] is countable because each of the sets on the right-hand side
  contains at most one element by~\eqref{eqAk1}. Hence
  \mbox{$A_k'''\cap[\beta+1,\omega_1)\setminus B_k^\beta$} is
  uncountable, and thus non-empty; that is, for each
  $\beta\in[0,\omega_1)$, we can find $\alpha\in
  A_k'''\cap[\beta+1,\omega_1)$ such that
  \mbox{$\min(F_{\alpha,k}\setminus\Delta_k) > \beta$}.

  Set $\zeta_1 =
  \sup\bigl(\{\zeta_0\}\cup\bigcup_{j\in\N}\Delta_j\bigr)\in
  [0,\omega_1)$.  A straightforward induction based on the above
  observation yields a strictly increasing transfinite sequence
  $(\alpha_\xi)_{\xi<\omega_1}$ in~$A_k'''\cap[\zeta_1+1,\omega_1)$
  such that $\sup\bigl(\{\zeta_1\}\cup\bigcup_{\eta<\xi}
  F_{\alpha_{\eta},k}\setminus\Delta_k\bigr) <
  \min(F_{\alpha_\xi,k}\setminus\Delta_k)$ for each
  $\xi\in[0,\omega_1)$, and consequently $A_k'''' = \{\alpha_\xi :
  \xi\in[0,\omega_1)\}$ is an uncountable subset
  of~$A_k'''\cap[\zeta_1+1,\omega_1)$ such that
  \begin{equation}\label{eqAk4}
    \sup\biggl(\{\zeta_1\}\cup\bigcup_{\beta\in A_k''''\cap[0,\alpha)}
    F_{\beta,k}\setminus\Delta_k\biggr) <
    \min(F_{\alpha,k}\setminus\Delta_k)\qquad (\alpha\in A_k'''').
  \end{equation}

  Set $m_k = |\Delta_k|<n_k$, and define $\lambda_k =
  \sum_{j=m_k+1}^{n_k} q_{j,k}\in\mathbb{L}$. Then,
  by~\eqref{eqDefnmualphak}, \eqref{eqAk5} and~\eqref{eqAk4}, we have
  $\lambda_k = \mu_{\alpha,k}([\zeta_1,\omega_1))$ for each $\alpha\in
  A_k''''$, and hence
  \begin{align*}
    1\ge\|\mu_{\alpha,k}\|\ge|\lambda_k|
    &\ge\bigl|\mu_{\alpha}([\zeta_1,\omega_1))\bigr| -
    \|\mu_{\alpha,k}-\mu_\alpha\|\\ &\ge
    \bigl|\mu_{\alpha}([\alpha,\omega_1))\bigr| -
    |\mu_{\alpha}|([\zeta_1,\alpha)) - \|\mu_{\alpha,k}-\mu_\alpha\| >
    \epsilon_0 - \frac{\epsilon_0}{3} - \frac{\epsilon_0}{3} =
    \frac{\epsilon_0}{3},
  \end{align*}
  so that, after passing to a subsequence, we may suppose that
  $(\lambda_k)_{k\in\N}$ is convergent with limit
  $\lambda\in\mathbb{K}$, say, where
  $1\ge|\lambda|\ge\epsilon_0/3>0$. (Note that, of all the estimates
  above, only \eqref{eqDefnmualphak2} depends explicitly on~$k$, and
  it clearly remains true after we pass to a sub\-sequence.)

  Suppose that $\Delta_k = \{\beta_{1,k},\ldots,\beta_{m_k,k}\}$,
  where $\beta_{1,k}<\cdots<\beta_{m_k,k}$, and define
  \begin{equation}\label{defnrhok} \rho_k =
    \sum_{j=1}^{m_k} q_{j,k}\delta_{\beta_{j,k}}\in
    C_0[0,\omega_1)^*. \end{equation} 
  Since $\|\rho_k\|\le \sum_{j=1}^{n_k}|q_{j,k}|
  = \|\mu_{\alpha,k}\|\le 1$ for each $\alpha\in A_k''''$,
  Lemma~\ref{dualballwsseqcomp} implies that, after replacing
  $(\rho_{k})_{k\in\N}$ with a subsequence, we may suppose that
  $(\rho_{k})_{k\in\N}$ is weakly$^*$ convergent with limit $\rho\in
  C_0[0,\omega_1)^*$, say.

  Our next aim is to show that, for each $(\alpha_k)_{k\in\N}$ which
  belongs to the set
  \begin{multline} \mathfrak{D} = \Bigl\{(\alpha_k)_{k\in\N} :
    \alpha_k\in A_k'''',\, \alpha_k<\alpha_{k+1}\ \text{and}\
    \max(F_{\alpha_k,k}\setminus \Delta_k) <
    \min(F_{\alpha_{k+1},k+1}\setminus \Delta_{k+1})\\ \text{for
      each}\ k\in\N,\ \text{and}\ \sup_{k\in\N}\alpha_k =
    \sup\bigcup_{k\in\N}F_{\alpha_k,k}\setminus
    \Delta_k\Bigr\},\label{defnscriptD} \end{multline} the sequence
  $(\mu_{\alpha_k})_{k\in\N}$ weakly$^*$ converges to
  $\rho+\lambda\delta_\alpha$, where $\alpha = \sup_{k\in\N}
  \alpha_k\in[0,\omega_1)$. By~\eqref{PKclaim1}, it suffices to show
  that $(\mu_{\alpha_k,k})_{k\in\N}$ weakly$^*$ converges to
  $\rho+\lambda\delta_\alpha$.  To verify this, let $\epsilon>0$ and
  $g\in C_0[0,\omega_1)$ be given. We may suppose that $\|g\|\le 1$.
  Choose $k_1\in\N$ such that $|\lambda-\lambda_k|<\epsilon/3$ and
  $|\langle g, \rho - \rho_k\rangle|<\epsilon/3$ whenever $k\ge k_1$.
  Since~$g$ is continuous at~$\alpha$, which is a limit ordinal, we
  can find $\beta_0\in[0,\alpha)$ such that $|g(\beta)-
  g(\alpha)|<\epsilon/3$ for each $\beta\in[\beta_0,\alpha]$.  By the
  definition of~$\mathfrak{D}$, we can take $k_2\in\N$ such that
  $F_{\alpha_k,k}\setminus\Delta_k\subseteq [\beta_0,\alpha]$
  when\-ever $k\ge k_2$, and thus $|g(\beta)- g(\alpha)|<\epsilon/3$
  for each $\beta\in\bigcup_{k\ge
    k_2}F_{\alpha_k,k}\setminus\Delta_k$. Now we have
  \begin{align*}
    \bigl|\langle g, \mu_{\alpha_k,k} - \rho -
    \lambda\delta_\alpha\rangle\bigr| &\le \bigl|\langle g,
    \mu_{\alpha_k,k} - \rho_k - \lambda_k\delta_\alpha\rangle\bigr| +
    \bigl|\langle g, \rho - \rho_k\rangle\bigr| + \bigl|\langle g,
    (\lambda - \lambda_k)\delta_\alpha\rangle\bigr|,
  \end{align*}
  where the second and third term are both less than $\epsilon/3$
  provided that $k\ge k_1$. To estimate the first term, we observe
  that $\theta_{\alpha_k,k}(j) = \beta_{j,k}$ for each
  $j\in\{1,\ldots,m_k\}$. Consequently~\eqref{eqDefnmualphak},
  \eqref{eqAk5} and~\eqref{defnrhok} imply that
  \[ \mu_{\alpha_k,k} - \rho_k =
  \sum_{j=m_k+1}^{n_k}q_{j,k}\delta_{\theta_{\alpha_k, k}(j)}, \]
  where $\theta_{\alpha_k,k}(j)\in F_{\alpha_k,k}\setminus\Delta_k$
  for each $j\in\{m_k+1,\ldots,n_k\}$, and therefore we have
  \begin{align*}
    \bigl|\langle g, \mu_{\alpha_k,k} - \rho_k -
    \lambda_k\delta_\alpha\rangle\bigr| &=
    \biggl|\sum_{j=m_k+1}^{n_k}q_{j,k}\bigl(g(\theta_{\alpha_k, k}(j))
    - g(\alpha)\bigr)\biggr|\\ &\le \sum_{j=m_k+1}^{n_k}|q_{j,k}|\cdot
    \bigl|g(\theta_{\alpha_k, k}(j)) - g(\alpha)\bigr| <
    \frac{\epsilon}{3}
  \end{align*}
  provided that $k\ge k_2$. Hence we conclude that
  $(\mu_{\alpha_k})_{k\in\N}$ weakly$^*$ converges to
  $\rho+\lambda\delta_\alpha$.

  This implies in particular that $\rho+\lambda\delta_\alpha\in K$ for
  each~$\alpha$ belonging to the set
  \begin{equation}\label{clubsetD5dec}
    D = \Bigl\{\sup_{k\in\N} \alpha_k :
    (\alpha_k)_{k\in\N}\in\mathfrak{D}\Bigr\}.
  \end{equation}
  We shall now complete the proof by showing that~$D$ is a club subset
  of~$[0,\omega_1)$. (Note that this will automatically ensure that
  $\rho\in K$ because the unboundedness of~$D$ implies that the net
  $(\delta_\alpha)_{\alpha\in D}$ converges weakly$^*$ to~$0$, so that
  $\rho = \text{w}^*\text{-}\lim_\alpha (\rho+\lambda\delta_\alpha)\in
  K$.)

  First, to see that $D$ is unbounded, let $\beta\in[0,\omega_1)$ be
  given.  By~\eqref{eqAk4}, the transfinite sequence
  $(\min(F_{\alpha,k}\setminus\Delta_{k}))_{\alpha\in A_k''''}$ is
  strictly increasing for each $k\in\N$, and thus unbounded. We can
  therefore inductively construct a sequence $(\alpha_k)_{k\in\N}$ in
  $[\beta,\omega_1)$ such that $\alpha_k\in A_k''''$~and
  \[ \max\bigl(\{\alpha_k\}\cup
  (F_{\alpha_k,k}\setminus\Delta_k)\bigr) <
  \min\bigl(\{\alpha_{k+1}\}\cup
  (F_{\alpha_{k+1},k+1}\setminus\Delta_{k+1})\bigr)\qquad (k\in\N). \]
  We claim that this sequence $(\alpha_k)_{k\in\N}$ belongs
  to~$\mathfrak{D}$.  Of the conditions in~\eqref{defnscriptD}, only
  the final one is not immediately obvious, and it follows from the
  intertwining relation
  \[ \alpha_k < \min(F_{\alpha_{k+1},k+1}\setminus\Delta_{k+1})\le
  \max(F_{\alpha_{k+1},k+1}\setminus\Delta_{k+1})<\alpha_{k+2}\qquad
  (k\in\N). \] Consequently, we have $\sup D\ge
  \sup_{k\in\N}\alpha_k\ge\beta$, as desired.

  Second, to verify that $D$ is closed, we observe that each
  $\beta\in\overline{D}$ is countable, and thus the limit of a
  sequence $(\beta^j)_{j\in\N}$ in~$D$.  We may suppose that
  $(\beta^j)_{j\in\N}$ is strictly increasing. For each $j\in\N$, take
  $(\alpha^j_k)_{k\in\N}\in\mathfrak{D}$ such that $\beta^j =
  \sup_{k\in\N}\alpha_k^j$. The sequence
  \mbox{$\bigl(\min(\{\alpha_k^{j}\}\cup
    (F_{\alpha^{j}_k,k}\setminus\Delta_k))\bigr)_{k\in\N}$} is then
  strictly increasing with limit~$\beta^j$, so we may inductively
  choose a strictly increasing sequence $(k_j)_{j\in\N}$ of integers
  such that $k_1 = 1$ and
  \begin{equation}\label{choosingkj}
    \beta^j < \min\bigl(\{\alpha_{k_{j+1}}^{j+1}\}\cup
    (F_{\alpha_{k_{j+1}}^{j+1},
      k_{j+1}}\setminus\Delta_{k_{j+1}})\bigr)\qquad (j\in\N).
  \end{equation}
  We now claim that the sequence $(\gamma_\ell)_{\ell\in\N}$ given by
  \[ (\alpha_1^1,\alpha_2^1,\ldots,\alpha_{k_2-1}^1, \alpha_{k_2}^2,
  \alpha_{k_2+1}^2,\ldots,\alpha_{k_3-1}^2,\alpha_{k_3}^3,\ldots,
  \alpha_{k_j-1}^{j-1},\alpha_{k_j}^j,\alpha_{k_j+1}^j, \ldots,
  \alpha_{k_{j+1}-1}^j,\alpha_{k_{j+1}}^{j+1},\ldots) \] belongs
  to~$\mathfrak{D}$.  Indeed, for $\ell\in\N$, let $j\in\N$ be the
  unique number such that \mbox{$k_j\le\ell< k_{j+1}$}. We then have
  $\gamma_\ell = \alpha_\ell^j\in A_\ell''''$.  If $\ell\ne
  k_{j+1}-1$, then $\gamma_{\ell+1} = \alpha_{\ell+1}^j$, in which
  case the inequalities $\gamma_\ell<\gamma_{\ell+1}$ and
  $\max(F_{\gamma_\ell,\ell}\setminus \Delta_\ell) <
  \min(F_{\gamma_{\ell+1},\ell+1}\setminus \Delta_{\ell+1})$ are both
  immediate from~\eqref{defnscriptD}. Otherwise $\ell=k_{j+1}-1$, and
  by~\eqref{choosingkj}, we find
  \[ \gamma_\ell = \alpha_\ell^j< \sup_{k\in\N}\alpha_k^j =
  \beta^j<\alpha_{k_{j+1}}^{j+1} = \gamma_{\ell+1} \] and
  \[ \max(F_{\gamma_\ell,\ell}\setminus \Delta_\ell) < \beta^j <
  \min(F_{\alpha_{k_{j+1}}^{j+1}, k_{j+1}}\setminus\Delta_{k_{j+1}}) =
  \min(F_{\gamma_{\ell+1},\ell+1}\setminus \Delta_{\ell+1}). \] These
  intertwining relations imply that $\sup_{\ell\in\N}\gamma_\ell =
  \beta = \sup\bigcup_{\ell\in\N}(F_{\gamma_\ell,\ell}\setminus
  \Delta_\ell)$, which shows that
  $(\gamma_\ell)_{\ell\in\N}\in\mathfrak{D}$, and hence $\beta\in D$,
  as required.

  \romanref{lemmaSimpleTopDich2},~$\Leftarrow$, now follows easily by
  contraposition, just as
  \romanref{lemmaSimpleTopDich3},~$\Leftarrow$, did. Indeed, suppose
  that condition~\alphref{lemmaSimpleTopDich1a} is not
  satisfied. Then, by~\romanref{lemmaSimpleTopDich1},
  condition~\alphref{lemmaSimpleTopDich1b} is satisfied, so that
  Lemma~\ref{copiesofw1} and the forward implication
  of~\romanref{lemmaSimpleTopDich3} imply that $K$ contains a subset
  which is homeo\-mor\-phic to~$[0,\omega_1]$. Hence $K$ is not
  (uniformly) Eberlein compact by Corollary~\ref{omega1notEC}.
\end{proof}

The idea that a result like Theorem~\ref{dichotomy} might be true was
inspired by a note~\cite{smith} from Richard Smith. The following
corollary confirms a conjecture that he proposed therein.

\begin{corollary} Let $K$ be a weakly$^*$ compact subset
  of~$C_0[0,\omega_1)^*$ such that there exists a continuous
  surjection from~$K$ onto~$[0,\omega_1]$. Then $K$ contains a
  homeomorphic copy of~$[0,\omega_1]$.
\end{corollary}
\begin{proof} By a classical result of Benyamini, Rudin and
  Wage~\cite{BRW}, the continuous image of an Eber\-lein compact space
  is Eberlein compact. Since~$[0,\omega_1]$ is not Eberlein compact,
  $K$~can\-not be Eberlein compact, and we are therefore in
  case~\romanref{dichotomy2} of Theorem~\ref{dichotomy}.
\end{proof}

\begin{example}\label{remark32}
  The purpose of this example is to show that the dichotomy stated in
  Lemma~\ref{lemmaSimpleTopDich}\romanref{lemmaSimpleTopDich1} is no
  longer true if condition~\romanref{lemmaSimpleTopDich1b} is replaced
  with the condition
  \begin{capromanenumerate}
  \item[(II$'$)] the set $S = \{ \alpha\in[0,\omega_1) :
    \mu([\alpha+1,\omega_1))\neq 0\ \text{\normalfont{for some}}\
    \mu\in K \}$ is stationary.
  \end{capromanenumerate}
  Indeed, let~$\Lambda$ be the set of all countable limit
  ordinals. Then \mbox{$K=\{\delta_\alpha-\delta_{\alpha+1}: \alpha
    \in \Lambda\}\cup\{0\}$} is a bounded and weakly$^*$ closed subset
  of~$C_0[0,\omega_1)^*$, and thus weakly$^*$ compact. Moreover, $K$
  satisfies condition~\romanref{lemmaSimpleTopDich1a}
  because~\eqref{lemmaSimpleTopDich2Eq1} holds for the club subset $D
  = \Lambda$ (and~$K$ is therefore uniformly Eberlein compact by
  Lemma~\ref{lemmaSimpleTopDich}\romanref{lemmaSimpleTopDich2}),
  but~$K$ also satisfies~(II$'$) because $\Lambda\subseteq S$, and
  each club subset~$E$ of~$[0,\omega_1)$ intersects~$\Lambda$, so that
  $S\cap E\neq\emptyset$. Hence
  conditions~\romanref{lemmaSimpleTopDich1a} and~(II$'$) are not
  mutually exclusive.
\end{example}

\begin{proposition}\label{prop111212}
  Every uniformly Eberlein compact space which contains a dense subset
  of cardinality at most~$\aleph_1$ is homeomorphic to a weakly$^*$
  compact subset of~$C_0[0,\omega_1)^*$.
\end{proposition}

\begin{proof}
  As above, let $\Lambda$ be the set of all countable limit
  ordinals. Then every uniformly Eberlein compact space containing a
  dense subset of cardinality at most~$\aleph_1$ embeds in the closed
  unit ball~$B_{\ell_2(\Lambda)}$ of the Hilbert space
  $\ell_2(\Lambda) = \bigl\{ f\colon \Lambda\to\mathbb{K} :
  \sum_{\alpha\in\Lambda} |f(\alpha)|^2<\infty\bigr\}$, equipped with
  the weak topology. Hence it will suffice to prove that the mapping
  given by
  \begin{equation}\label{prop111212eq1}
    \theta\colon\ f\mapsto
    \sum_{\alpha\in\Lambda}f(\alpha)|f(\alpha)|(\delta_\alpha -
    \delta_{\alpha+1}),\quad B_{\ell_2(\Lambda)}\to C_0[0,\omega_1)^*,
  \end{equation}
  is a weakly-weakly$^*$ continuous injection. The injectivity is
  clear.  Suppose that the net $(f_j)_{j\in J}$
  in~$B_{\ell_2(\Lambda)}$ converges weakly to~$f$, and let
  $\epsilon>0$ and $g\in C_0[0,\omega_1)$ be given. Since the
  indicator functions $\mathbf{1}_{[0,\alpha]}$ for
  $\alpha\in[0,\omega_1)$ span a norm-dense subspace
  of~$C_0[0,\omega_1)$, it suffices to consider the case where $g =
  \mathbf{1}_{[0,\alpha]}$ for some $\alpha\in[0,\omega_1)$. Now
  \[ \bigl\langle \mathbf{1}_{[0,\alpha]},\theta(h)\bigr\rangle
  = \begin{cases} h(\alpha)|h(\alpha)|\ &\text{if}\ \alpha\in\Lambda\\
    0\ &\text{otherwise} \end{cases}\qquad (h\in\ell_2(\Lambda)), \]
  so that we may suppose that $\alpha\in\Lambda$. Choosing $j_0\in J$
  such that $|f(\alpha) - f_j(\alpha)|<\epsilon/2$ whenever $j\ge
  j_0$, we obtain
  \begin{align*} \bigl|\bigl\langle \mathbf{1}_{[0,\alpha]},\theta(f) -
    \theta(f_j)\bigr\rangle\bigr|&= \bigl| f(\alpha)|f(\alpha)| -
    f_j(\alpha)|f_j(\alpha)|\bigr|\\ &\le
    2\max\bigl\{|f(\alpha)|,|f_j(\alpha)|\bigr\}|f(\alpha) -
    f_j(\alpha)| < \epsilon\qquad (j\ge j_0), \end{align*} which
  proves that $(\theta(f_j))_{j\in J}$ converges weakly$^*$
  to~$\theta(f)$.
\end{proof}

\begin{remark}
  The mapping $\theta\colon \ell_2(\Lambda)\to C_0[0,\omega_1)^*$
  given by~\eqref{prop111212eq1} is clearly not linear. In fact, no
  weakly-weakly$^*$ continuous, linear mapping $T\colon
  \ell_2(\Lambda)\to C_0[0,\omega_1)^*$ is injective. To verify this,
  we first observe that~$T$ is weakly compact because its domain is
  reflexive, and hence compact because its codomain has the Schur
  property. Moreover, using the reflexivity of~$\ell_2(\Lambda)$ once
  more, we see that the weak-weak$^*$ continuity of~$T$ implies that
  $T = S^*$ for some operator $S\colon
  C_0[0,\omega_1)\to\ell_2(\Lambda)$. Schauder's theorem then shows
  that~$S$ is compact, so that it has separable range. In particular,
  the range of~$S$ is not dense in~$\ell_2(\Lambda)$, and therefore
  $T=S^*$ is not injective.
\end{remark}

\section{Operator theory on $C_0[0,\omega_1)$}\label{section4}
\noindent
The following lemma represents the core of our proof of
Theorem~\ref{surjective}.

\begin{lemma}\label{lemma0102012}
  Let~$X$ be a Banach space, and suppose that there exists a
  surjective operator $T\colon C_0[0,\omega_1)\to X$. Then exactly one
  of the following two alternatives holds:
  \begin{capromanenumerate}
  \item\label{lemma0102012i} either $X$ embeds in a
    Hilbert-generated Banach space; or
  \item\label{lemma0102012ii} there exists a club subset~$D$
    of~$[0,\omega_1)$ such that the restriction of~$T$ to the
    subspace~$\mathscr{R}_D$ given by~\eqref{phiDlemma07082012eq3} is
    bounded below, and $T[\mathscr{R}_D]$ is complemented in~$X$.
  \end{capromanenumerate}
\end{lemma}

\begin{proof} Let $B_{X^*}$ be the closed unit ball of~$X^*$.  The
  weak$^*$ continuity of~$T^*$ implies that the subset
  $K=T^*[B_{X^*}]$ of~$C_0[0,\omega_1)^*$ is weakly$^*$ compact, so by
  Theorem~\ref{dichotomy}, we have:
  \begin{capromanenumerate}
  \item\label{lemma0102012I} either $K$ is uniformly Eberlein compact; or
  \item\label{lemma0102012II} there exist $\rho\in K$,
    $\lambda\in\mathbb{K}\setminus\{0\}$, and a club subset~$D$
    of~$[0,\omega_1)$ such that $\rho+\lambda\delta_\alpha\in K$ for
    each $\alpha\in D$.
  \end{capromanenumerate}
  Since $T^*$ is injective by the assumption, its restriction
  to~$B_{X^*}$ is a weak$^*$ homeomorphism onto~$K$. Hence, in the
  first case, $B_{X^*}$ is also uniformly Eberlein compact, and so~$X$
  embeds in a Hilbert-generated Banach space by
  Theorem~\ref{charUECandHG}.

  Otherwise we can choose functionals $g,g_\alpha\in B_{X^*}$ such
  that $T^*g=\rho$ and $T^*g_\alpha=\lambda\delta_\alpha+\rho$ for
  each $\alpha\in D$. Given $x\in X$, we define a mapping
  $Sx\colon[0,\omega_1)\to\mathbb{K}$ by
  \begin{equation}\label{lemma0102012eq1}
    (Sx)(\alpha) = \langle x, g_{\pi_D(\alpha)} - g\rangle\qquad 
    (\alpha\in[0,\omega_1)), 
  \end{equation}
  where $\pi_D\colon [0,\omega_1)\to D$ is the retraction defined
  by~\eqref{phiDlemma07082012iiEq1}. Suppose that $x = Tf$ for some
  $f\in C_0[0,\omega_1)$. Then, for each $\alpha\in[0,\omega_1)$, we
  have
  \begin{align}\label{lemma0102012eq2} (Sx)(\alpha) &= \langle Tf,
    g_{\pi_D(\alpha)} - g\rangle = \langle f, T^*g_{\pi_D(\alpha)} -
    T^*g\rangle = \lambda f(\pi_D(\alpha)) = \lambda (P_Df)(\alpha),
  \end{align}
  so that $Sx = \lambda P_Df$, where $P_D$ is the projection defined
  in
  Lemma~\ref{phiDlemma07082012}\romanref{phiDlemma07082012iv}. Since
  $\lambda P_Df\in C_0[0,\omega_1)$, we see
  that~\eqref{lemma0102012eq1} defines a mapping $S\colon X\to
  C_0[0,\omega_1)$, which is linear by the linearity of the
  functionals $g_{\pi_D(\alpha)}$ and~$g$.  Moreover, $S$ is bounded
  because the Open Mapping Theorem implies that there exists a
  constant $C>0$, dependent only on the surjective operator~$T$, such
  that, for each $x\in X$, there exists $f\in C_0[0,\omega_1)$ with
  $\| f\|\le C\|x\|$ and $Tf = x$.  Then we have $Sx = \lambda P_Df$
  by~\eqref{lemma0102012eq2}, and hence $\| Sx\| = |\lambda|\, \|
  P_Df\|\le |\lambda|C\|x\|$, as desired.

  Another application of~\eqref{lemma0102012eq2} shows that $STP_D =
  \lambda P_D$ because
  \[ (STP_D)f = S(T(P_Df)) = \lambda P_D(P_Df) = \lambda P_Df\qquad
  (f\in C_0[0,\omega_1)). \] It is now straightforward to verify that
  $T|_{\mathscr{R}_D}$ is bounded below by $|\lambda|/\|S\|$ and that
  the operator $\lambda^{-1}TP_DS\in\mathscr{B}(X)$ is a projection
  with range~$T[\mathscr{R}_D]$.

  Finally, the two conditions are mutually exclusive
  because~\romanref{lemma0102012ii} implies that~$T$ induces an
  isomorphic embedding of~$\mathscr{R}_D\cong C_0[0,\omega_1)$ in~$X$,
  and hence~\romanref{lemma0102012i} cannot be satisfied by
  Theorem~\ref{w1notUEC}.
\end{proof}

\begin{proof}[Proof of Theorem~{\normalfont{\ref{surjective}}}] 
  Suppose that~\romanref{surjective1} is not satisfied. Then, by
  Lemma~\ref{lemma0102012}, we can find a club subset~$D$
  of~$[0,\omega_1)$ and an idempotent operator~$Q$ on~$X$ such that
  $T|_{\mathscr{R}_D}$ is bounded below and $Q[X] = T[\mathscr
  R_D]$. Hence the identity operator on~$C_0[0,\omega_1)$ factors
  through~$T$ by Lemmas~\ref{operatorsfactoringID}
  and~\ref{phiDlemma07082012}\romanref{phiDlemma07082012iv}.

  We shall complete the proof that~\romanref{surjective2} is satisfied
  by showing that $\ker Q$ embeds in a Hilbert-generated Banach space.
  Assume the contrary, and apply Lemma~\ref{lemma0102012} to the
  surjective operator $U\colon f\mapsto (I_X - Q)T f,\,
  C_0[0,\omega_1)\to \ker Q$, to obtain a club subset~$E$
  of~$[0,\omega_1)$ such that $U|_{\mathscr{R}_E}$ is bounded below by
  $\epsilon>0$, say.  Then we have
  \[ \epsilon\|f\|\le \| Uf\| = \bigl\|(I_X - Q)T f\bigr\| = 0\qquad
  (f\in \mathscr{R}_D\cap\mathscr{R}_E), \] so that $\mathscr{R}_D\cap
  \mathscr{R}_E = \{0\}$. This, however, contradicts
  Lemma~\ref{intersections}.

  Theorem~\ref{w1notUEC} shows that conditions~\romanref{surjective1}
  and~\romanref{surjective2} are mutually exclusive.
\end{proof}

\begin{remark} Not all quotients of~$C_0[0,\omega_1)$ are subspaces
  of~$C_0[0,\omega_1)$.  This follows from a result of
  Alspach~\cite{alspachquotients}, which says
  that~$C[0,\omega^\omega]$ has a quotient~$X$ which does not embed
  in~$C[0,\alpha]$ for any countable ordinal~$\alpha$.
  Since~$C_0[0,\omega_1)$ contains a complemented copy
  of~$C[0,\omega^\omega]$, $X$ is also a quotient
  of~$C_0[0,\omega_1)$. However, $X$ does not embed
  in~$C_0[0,\omega_1)$ because $X$ is separable (being a quotient of a
  separable space), and each separable subspace of~$C_0[0,\omega_1)$
  embeds in~$C[0,\alpha]$ for some countable ordinal~$\alpha$
  (\emph{e.g.}, see \cite[Lemma~4.2]{kanialaustsen}).
\end{remark} 

\begin{lemma}\label{dichotomyw1} 
  Exactly one of the following two alternatives holds for each
  continuous mapping~$\theta$ from~$[0,\omega_1]$ into a Hausdorff
  space~$K\colon$
  \begin{capromanenumerate}
  \item\label{dichotomyw1i} either $[0,\omega_1]$ contains a closed,
    uncountable subset~$D$ such that $\theta|_D$ is constant; or
  \item\label{dichotomyw1ii} $[0,\omega_1]$ contains a closed,
    uncountable subset~$E$ such that $\theta|_E$ is injective, and
    hence a homeomorphism onto~$\theta[E]$.
\end{capromanenumerate}
\end{lemma}
\begin{proof} Suppose that~$K$ contains a point~$x$ whose pre-image
  under~$\theta$ is uncountable. Then~\romanref{dichotomyw1i} is
  satisfied for $D = \theta^{-1}[\{x\}]$.

  Otherwise each point of~$K$ has countable pre-image
  under~$\theta$. In this case we shall inductively construct a
  strictly increasing transfinite sequence
  $(\alpha_\xi)_{\xi<\omega_1}$ in~$[1,\omega_1)$ such that
  \begin{equation}\label{dichotomyw1Eq1}
    \{ \alpha\in[0,\omega_1) :  \theta(\alpha) = \theta(\omega_1)\ 
    \text{or}\  \theta(\alpha) = \theta(\alpha_\eta)\ 
    \text{for some}\ \eta<\xi\}\subseteq 
    [0,\alpha_\xi)\quad (\xi\in[0,\omega_1))
  \end{equation} 
  and
  \begin{equation}\label{dichotomyw1Eq2}
    \sup_{\eta<\xi}\alpha_\eta\in\{\alpha_\eta:\eta\le\xi\}\qquad 
    (\xi\in[1,\omega_1)). 
  \end{equation} 

  To start the induction, let $\alpha_0 = \sup\{\alpha\in[0,\omega_1)
  : \theta(\alpha) = \theta(\omega_1)\}+1\in[1,\omega_1)$.

  Now assume inductively that, for some $\xi\in[1,\omega_1)$, a
  strictly increasing sequence $(\alpha_\eta)_{\eta<\xi}$ of countable
  ordinals has been chosen in accordance with~\eqref{dichotomyw1Eq1}
  and~\eqref{dichotomyw1Eq2}, and define
  \[ \alpha_\xi = \sup\{\alpha + 1 : \theta(\alpha) =
  \theta(\alpha_\eta)\ \text{for some}\ \eta<\xi\}\in
  \Bigl[\sup_{\eta<\xi}(\alpha_\eta+1),\omega_1\Bigr). \]
  Then~\eqref{dichotomyw1Eq1} is certainly satisfied for~$\xi$.  To
  verify~\eqref{dichotomyw1Eq2}, let $\beta =
  \sup_{\eta<\xi}\alpha_\eta$. The conclusion is clear if this
  supremum is attained.  Otherwise $\xi$ is a limit ordinal, and we
  claim that $\beta =\alpha_\xi$. Since $\beta\le\alpha_\xi$, it
  suffices to show that $\alpha+1\le\beta$ whenever
  $\alpha\in[0,\omega_1)$ satisfies $\theta(\alpha) =
  \theta(\alpha_\eta)$ for some $\eta<\xi$. Now $\eta+1<\xi$ because
  $\xi$ is a limit ordinal, so that~\eqref{dichotomyw1Eq1} holds for
  $\eta+1$ by the induction hypothesis. Hence
  $\alpha+1\le\alpha_{\eta+1} < \beta$, and the induction~continues.

  Let $E = \{\alpha_\xi: \xi<\omega_1\}\cup\{\omega_1\}$, which is
  uncountable because $(\alpha_\xi)_{\xi<\omega_1}$ is strictly
  increasing, and closed by~\eqref{dichotomyw1Eq2}. Moreover,
  $\theta|_E$ is injective by~\eqref{dichotomyw1Eq1}, so
  that~\romanref{dichotomyw1ii} is satisfied.

  Finally, to see that conditions~\romanref{dichotomyw1i}
  and~\romanref{dichotomyw1ii} are mutually exclusive, assume towards
  a contradiction that $[0,\omega_1]$ contains closed, uncountable
  subsets~$D$ and~$E$ such that $\theta|_D$ is constant and
  $\theta|_E$ is injective. Then $D\cap E$ is uncountable and
  contains~$\omega_1$, so that $\theta(\alpha) = \theta(\omega_1)$ for
  each $\alpha\in D\cap E$ by the choice of~$D$. This, however,
  contradicts the injectivity of~$\theta|_E$.
\end{proof}

\begin{proof}[Proof of Theorem~{\normalfont{\ref{clubandscalar}}}] 
  Let $T\in\mathscr{B}(C_0[0,\omega_1))$. The following
  ``function-free'' reformulation of~\eqref{clubandscalarEq1} will
  enable us to simplify certain calculations somewhat:
  \begin{equation}\label{clubandscalarEq2}
    T^*\delta_\alpha = \phi(T)\delta_\alpha\qquad (\alpha\in D).
  \end{equation}
  To prove that there exist~$\phi(T)\in\mathbb{K}$ and a club
  subset~$D$ of~$[0,\omega_1)$ such that this identity is satisfied,
  we consider the composite mapping~$\theta$ given by
  \[ \spreaddiagramrows{2ex}\spreaddiagramcolumns{.5ex}%
  \xymatrix{ [0,\omega_1]\ar@{-->}^-{\displaystyle{\theta}}[rr]%
    \ar_-{\displaystyle{\sigma_{1,0}}}[dr] & & C_0[0,\omega_1)^*\\ &
    C_0[0,\omega_1)^*\ar_-{\displaystyle{T^*}}[ur]\smashw[l]{,}} \]
  where both copies of~$C_0[0,\omega_1)^*$ are equipped with the
  weak$^*$ topology, and $\sigma_{1,0}$ is the injection defined in
  Lemma~\ref{copiesofw1}, that is, $\sigma_{1,0}(\alpha) =
  \delta_\alpha$ for $\alpha\in[0,\omega_1)$ and
  $\sigma_{1,0}(\omega_1) = 0$. Since~$\sigma_{1,0}$ and~$T^*$ are
  continuous, so is~$\theta$. We may therefore apply
  Lemma~\ref{dichotomyw1} to conclude that~$[0,\omega_1]$ contains a
  closed, uncountable subset~$E$ such that either $\theta|_E$ is
  constant, or $\theta|_E$ is injective. Note that $\omega_1\in E$,
  and $E\setminus\{\omega_1\}$ is a club subset of~$[0,\omega_1)$.

  If $\theta|_E$ is constant, then we have
  \[ T^*\delta_\alpha = (T^*\circ\sigma_{1,0})(\alpha) =
  \theta(\alpha) = \theta(\omega_1) = (T^*\circ\sigma_{1,0})(\omega_1)
  = T^*0 = 0\qquad \bigl(\alpha\in E\setminus\{\omega_1\}\bigr), \] so
  that~\eqref{clubandscalarEq2} is satisfied for $D =
  E\setminus\{\omega_1\}$ and $\phi(T) = 0$.

  Otherwise $\theta|_E$ is injective, in which case
  $\widetilde{\theta}\colon\alpha\mapsto\theta(\alpha),\,
  E\to\theta[E]$, is a homeomorphism.  Since~$E$ is homeomorphic
  to~$[0,\omega_1]$, which is not Eberlein compact,
  Theorem~\ref{dichotomy} implies that there exist $\rho\in\theta[E]$,
  $\phi(T)\in\mathbb{K}\setminus\{0\}$ and a club subset~$F$
  of~$[0,\omega_1)$ such that $\rho +
  \phi(T)\delta_\alpha\in\theta[E]$ for each $\alpha\in F$; let
  $\eta_\alpha = \widetilde{\theta}^{-1}(\rho +
  \phi(T)\delta_\alpha)$. Then $(\eta_\alpha)_{\alpha\in F}$ is an
  un\-countable net of distinct elements of~$E$, and it converges
  to~$\widetilde{\theta}^{-1}(\rho)$. The only possible limit of such
  a net is~$\omega_1$, so that $\rho = \theta(\omega_1) = 0$, and we
  have
  \begin{equation}\label{clubandscalarEq3}
    T^*\delta_{\eta_\alpha} = \theta(\eta_\alpha) =
    \phi(T)\delta_\alpha\qquad (\alpha\in F).
  \end{equation}

  We shall now show that $D = \{\alpha\in[0,\omega_1) :
  T^*\delta_\alpha = \phi(T)\delta_\alpha\}$ is a club subset
  of~$[0,\omega_1)$; this will complete the existence part of the
  proof because~\eqref{clubandscalarEq2} is evidently satisfied for
  this choice of~$D$.

  Suppose that $(\alpha_j)_{j\in J}$ is a net in~$D$ converging
  to~$\alpha\in[0,\omega_1)$. The net $(\delta_{\alpha_j})_{j\in J}$
  is then weakly$^*$ convergent with limit~$\delta_\alpha$. Hence, by
  the weak$^*$ continuity of~$T^*$, we have
  \[ T^*\delta_\alpha = \text{w}^*\text{-}\lim_j T^*\delta_{\alpha_j}
  = \text{w}^*\text{-}\lim_j \phi(T)\delta_{\alpha_j} =
  \phi(T)\delta_\alpha, \] so that $\alpha\in D$, which proves that
  $D$ is closed.

  To prove that~$D$ is unbounded, let $\gamma\in[0,\omega_1)$.  For
  each $\beta\in[0,\omega_1)$, the sets $F\cap[0,\beta]$ and
  $\{\alpha\in F : \eta_\alpha\le\beta\}$ are both countable, so that
  the complement of their union, which is equal to $\{\alpha\in
  F\cap[\beta+1,\omega_1) : \eta_\alpha > \beta\}$, is uncountable,
  and thus non-empty. Using this, we can inductively construct a
  sequence $(\alpha_n)_{n\in\N}$ in~$F\cap[\gamma+1,\omega_1)$ such
  that
  \[ \max\{\alpha_n,\eta_{\alpha_n}\} <
  \min\{\alpha_{n+1},\eta_{\alpha_{n+1}}\}\qquad (n\in\N). \] Let
  $\alpha = \sup_{n\in\N}\alpha_n\in F\cap[\gamma+1,\omega_1)$.  Then
  both of the sequences $(\alpha_n)_{n\in\N}$ and
  $(\eta_{\alpha_n})_{n\in\N}$ converge to~$\alpha$, so that
  $(\delta_{\alpha_n})_{n\in\N}$ and
  $(\delta_{\eta_{\alpha_n}})_{n\in\N}$ both weakly$^*$ converge
  to~$\delta_\alpha$, and consequently
  \[ T^*\delta_\alpha = \text{w}^*\text{-}\lim_n
  T^*\delta_{\eta_{\alpha_n}} = \text{w}^*\text{-}\lim_n
  \phi(T)\delta_{\alpha_n} = \phi(T)\delta_\alpha \] by the weak$^*$
  continuity of~$T^*$ and~\eqref{clubandscalarEq3}, so that $\alpha\in
  D$, which is therefore unbounded.

  We shall next prove that the scalar~$\phi(T)$ is uniquely determined
  by the operator~$T$.  Suppose that~$\phi_1(T)$ and~$\phi_2(T)$ are
  scalars such that
  \[ (Tf)(\alpha) = \phi_j(T)f(\alpha)\qquad (\alpha\in D_j,\, f\in
  C_0[0,\omega_1),\, j=1,2) \] for some club subsets ~$D_1$ and~$D_2$
  of~$[0,\omega_1)$.  Then $D_1\cap D_2$ is a club subset, and hence
  non-empty.  Taking $\alpha\in D_1\cap D_2$, we obtain
  \[ (T\mathbf{1}_{[0,\alpha]})(\alpha) =
  \phi_j(T)\mathbf{1}_{[0,\alpha]}(\alpha) = \phi_j(T)\qquad
  (j=1,2), \] so that $\phi_1(T) = \phi_2(T)$, as required.

  Consequently, we can define a mapping $\phi\colon
  T\mapsto\phi(T),\,\mathscr{B}(C_0[0,\omega_1))\to\mathbb{K}$, which
  is non-zero because $\phi(I_{C_0[0,\omega_1)}) =1$. To see that
  $\phi$ is an algebra homomorphism, let $\lambda\in\mathbb{K}$ and
  $T_1,T_2\in\mathscr{B}(C[0,\omega_1))$ be given, and take club
  subsets~$D_1$ and~$D_2$ of~$[0,\omega_1)$ such that
  \[ T_j^*\delta_\alpha = \phi(T_j)\delta_\alpha\qquad (\alpha\in
  D_j,\,j=1,2). \] Then, for each~$\alpha$ belonging to the club
  subset~$D_1\cap D_2$ of~$[0,\omega_1)$, we have
  \[ (\lambda T_1+T_2)^*\delta_\alpha = \lambda T_1^*\delta_\alpha +
  T_2^*\delta_\alpha = (\lambda\phi(T_1) + \phi(T_2))\delta_\alpha \] and
  \[ (T_1T_2)^*\delta_\alpha = T_2^*(T_1^*\delta_\alpha) =
  T_2^*(\phi(T_1)\delta_\alpha) = \phi(T_1)\phi(T_2)\delta_\alpha, \]
  so that $\phi(\lambda T_1+T_2) = \lambda\phi(T_1) + \phi(T_2)$ and
  $\phi(T_1T_2) = \phi(T_1)\phi(T_2)$ by the definition of~$\phi$.
\end{proof}

\begin{lemma}\label{lemma25102012} Let~$T$ be an operator
  on~$C_0[0,\omega_1)$ such that
  \begin{equation}\label{lemma25102012Eq1} (Tf)(\alpha) = 0\qquad
    (f\in C_0[0,\omega_1),\,\alpha\in D) \end{equation}
  for some club subset~$D$ of~$[0,\omega_1)$. Then the range of~$T$ is
  contained  in  the  kernel  of the  projection~$P_D$  introduced  in
  Lemma~{\normalfont{\ref{phiDlemma07082012}}}%
  \romanref{phiDlemma07082012iv}.
\end{lemma}
\begin{proof}
  This is immediate from~\eqref{phiDlemma07082012eq2}
  and~\eqref{lemma25102012Eq1}.
\end{proof}

\begin{proof}[Proof of Theorem~{\normalfont{\ref{thmcharloywillis}}}]
  We begin by showing that conditions~\alphref{lambdazero}
  and~\alphref{character0} are equivalent. This relies on the fact
  that, by~\eqref{Rudinmatrix}, we have
  \begin{equation}\label{thmcharloywillisEq1}
    T_{\alpha,\beta} = T^*\delta_\alpha(\{\beta\})\qquad
    (\alpha,\beta\in[0,\omega_1)). \end{equation}

  To see that~\alphref{lambdazero} implies~\alphref{character0},
  suppose that~$D$ is a club subset of~$[0,\omega_1)$ such that
  $T_{\alpha,\alpha} = 0$ for each $\alpha\in D$, and take a club
  subset~$E$ of~$[0,\omega_1)$ such that $T^*\delta_\alpha =
  \phi(T)\delta_\alpha$ for each $\alpha\in E$. Then, choosing
  $\alpha\in D\cap E$ (which is possible because $D\cap E$ is a club
  subset, and hence non-empty), we obtain
  \[ 0 = T_{\alpha,\alpha} = T^*\delta_\alpha(\{\alpha\}) =
  \phi(T)\delta_\alpha(\{\alpha\}) = \phi(T), \] which shows
  that~\alphref{character0} is satisfied.

  Conversely, suppose that $\phi(T) = 0$, so that there exists a club
  subset~$D$ of~$[0,\omega_1)$ such that $T^*\delta_\alpha = 0$ for
  each~$\alpha\in D$.  Then~\eqref{thmcharloywillisEq1} immediately
  shows that $T_{\alpha,\alpha} = 0$ for each $\alpha\in D$, so
  that~\alphref{lambdazero} is satisfied.

  We shall next prove that
  conditions~\alphref{character0}--\alphref{Identity} are equivalent.

  The implication \alphref{character0}$\Rightarrow$\alphref{Lzero}
  follows immediately from Lemma~\ref{lemma25102012} and
  Corollary~\ref{OmegaXiLemma10Jan2013ii}.

  \alphref{Lzero}$\Rightarrow$\alphref{tishg}. Suppose that we have $T
  = SR$ for some operators $R\colon C_0[0,\omega_1)\to C_0(L_0)$ and
  \mbox{$S\colon C_0(L_0)\to C_0[0,\omega_1)$}. Then
  $R\in\mathscr{HG}(C_0[0,\omega_1), C_0(L_0))$ by Corollary
  \ref{L0uniformE}, and hence $T\in\mathscr{HG}(C_0[0,\omega_1))$ by
  Lemma~\ref{hilbertideal}\romanref{hilbertideal1}.

  The implication \alphref{tishg}$\Rightarrow$\alphref{tiswcg} is
  clear because each Hilbert-generated Banach space is weakly
  compactly generated.

  The implications
  \alphref{tiswcg}$\Rightarrow$\alphref{Tdoesntfix}$\Rightarrow$%
  \alphref{Identity}$\Rightarrow$\alphref{character0} are all proved
  by contraposition.

  \alphref{tiswcg}$\Rightarrow$\alphref{Tdoesntfix}. Suppose that $TU$
  is an isomorphism onto its range for some operator~$U$
  on~$C_0[0,\omega_1)$. Then, by Theorem~\ref{w1notUEC}, the range
  of~$TU$, and hence the range of~$T$, cannot be contained in any
  weakly compactly generated Banach space.

  \alphref{Tdoesntfix}$\Rightarrow$\alphref{Identity}. Suppose that
  $STR = I_{C_0[0,\omega_1)}$ for some operators~$R$ and~$S$
  on~$C_0[0,\omega_1)$. The operator~$TR$ is then bounded below, and
  hence an isomorphism onto its range, so that $T$ fixes a copy
  of~$C_0[0,\omega_1)$.

  \alphref{Identity}$\Rightarrow$\alphref{character0}.  Suppose that
  $\phi(T)\neq 0$. By rescaling, we may suppose that $\phi(T) = 1$, so
  that~\eqref{clubandscalarEq1} is satisfied for some club subset~$D$
  of~$[0,\omega_1)$. Then, in the notation of
  Lemma~\ref{phiDlemma07082012}, we have a commutative diagram, which
  implies that~\alphref{Identity} is not satisfied:
  \[ \spreaddiagramrows{2ex}\spreaddiagramcolumns{2ex}%
  \xymatrix{
    C_0[0,\omega_1)\ar^-{\displaystyle{I_{C_0[0,\omega_1)}}}[rrr]
    \ar@{-->}[dd] \ar_-{\displaystyle{U_D^{-1}}}[rd] & & &
    C_0[0,\omega_1)\\ & C_0(D)\ar^-{\displaystyle{I_{C_0(D)}}}[r]%
    \ar_-{\displaystyle{S_D}}[dl] &
    C_0(D)\ar_-{\displaystyle{U_D}}[ur]\\
    C_0[0,\omega_1)\ar^-{\displaystyle{T}}[rrr] & & &
    C_0[0,\omega_1)\ar_-{\displaystyle{R_D}}[ul]\ar@{-->}[uu]\smashw[l]{.}} \]
  Indeed, the commutativity of the upper trapezium is clear, while for
  the lower one, we find
  \[ (R_DTS_Dg)(\alpha) = 
  T(g\circ\pi_D)(\alpha) = (g\circ\pi_D)(\alpha) = g(\alpha)\qquad
  (g\in C_0(D),\,\alpha\in D). \] 

  Finally, to see that conditions~\alphref{TinLW}
  and~\alphref{character0} are equivalent, we note that, on the one
  hand, the Loy--Willis ideal~$\mathscr{M}$ is a maximal ideal
  of~$\mathscr{B}(C_0[0,\omega_1))$ by its definition. On the other,
  the implication \alphref{Identity}$\Rightarrow$\alphref{character0},
  which has just been established, shows that the identity operator
  belongs to the ideal generated by any operator not in~$\ker\phi$, so
  that $\ker\phi$ is the \textsl{unique} maximal ideal
  of~$\mathscr{B}(C_0[0,\omega_1))$. Hence $\ker\phi = \mathscr{M}$.
\end{proof}

As we have already noted in the final paragraph of the proof above,
Theorem~\ref{thmcharloywillis} has the following important
consequence, which generalizes \cite[Theorem~1.1]{kanialaustsen}.

\begin{corollary}\label{kerphiuniquemaxideal}
The ideal $\mathscr{M} = \ker\phi = \mathscr{HG}(C_0[0,\omega_1))$ is
  the unique maximal ideal of~$\mathscr{B}(C_0[0,\omega_1))$.
\end{corollary}

\begin{proof}[Proof of Corollary~{\normalfont{\ref{blai}}}]
  Let~$\Gamma$ denote the set of all club subsets of~$[0,\omega_1)$,
  ordered by reverse inclusion. This order is filtering upward because
  $D\cap E\in\Gamma$ is a majorant for any pair $D,E\in\Gamma$, and
  hence Theorem~\ref{thmcharloywillis},
  Lemma~\ref{phiDlemma07082012}\romanref{phiDlemma07082012iv} and
  Corollary~\ref{OmegaXiLemma10Jan2013ii} imply that \[ Q_D =
  I_{C_0[0,\omega_1)} - P_D\in\mathscr{M}\qquad (D\in\Gamma) \]
  defines a net of projections, each having norm at most two.  For
  each $T\in \mathscr{M}$, we have $P_{D}T = 0$ for some club
  subset~$D$ of~$[0,\omega_1)$ by
  Lemma~\ref{lemma25102012}. Equation~\eqref{phiDlemma07082012eq2}
  then shows that $P_ET = 0$ for each $E\subseteq D$; that is, $Q_ET =
  T$ whenever $E\ge D$.
\end{proof}

\begin{example}\label{BLAIexample}
  Consider the Hilbert space \mbox{$H = \{ f\colon
    [0,\omega_1)\to\mathbb{K} : \sum_{\alpha<\omega_1}
    |f(\alpha)|^2<\infty\}$}.  The work of Gramsch~\cite{gramsch} and
  Luft~\cite{luft} shows that the set $\mathscr{X}(H)$ of operators
  on~$H$ having separable range is the unique maximal ideal
  of~$\mathscr{B}(H)$. (In fact, Gramsch and Luft proved that the
  entire lattice of closed ideals of~$\mathscr{B}(H)$ is given by \[
  \{0\}\subset \mathscr{K}(H)\subset \mathscr{X}(H)\subset
  \mathscr{B}(H), \] but we do not require the full strength of their
  result.)  Since~$\mathscr{B}(H)$ is a $C^*$-algebra, each of its
  closed ideals has a bounded two-sided approximate identity
  consisting of positive contractions. The purpose of this example is
  to show that, in the case of $\mathscr{X}(H)$, we have a bounded
  two-sided approximate identity $(P_L)_{L\in\Gamma}$ consisting of
  contractive, self-adjoint projections such that $P_LT=T = TP_L$
  eventually for each $T\in\mathscr{X}(H)$. We note in passing that
  algebras which contain a net with this property have been studied in
  a purely algebraic context by Ara and
  Perera~\cite[Defini\-tion~1.4]{araP} and Pedersen and
  Perera~\cite[Sec\-tion~4]{pereraPed}.

  Let $\Gamma$ denote the set of all closed, separable subspaces
  of~$H$, ordered by inclusion. This order is filtering upward because
  $\overline{L+M}\in\Gamma$ majorizes the pair $L,M\in\Gamma$. For
  \mbox{$L\in\Gamma$}, let $P_L\in\mathscr{X}(H)$ be the orthogonal
  projection which has range~$L$. Suppose that
  \mbox{$T\in\mathscr{X}(H)$}, and denote by~$T^\star$ the
  Hilbert-space adjoint of~$T$. We have $T^\star\in\mathscr{X}(H)$
  because each closed ideal of a $C^*$-algebra is self-adjoint, and
  there\-fore $M = \overline{T[H] + T^\star[H]}$ belongs to~$\Gamma$.
  Now, for each $L\in\Gamma$ such that $L\supseteq M$, we see that
  $P_LT = T$ and $P_LT^\star = T^\star$, from which the desired
  conclusion follows by taking the adjoint of the latter equation.
\end{example}

\begin{proof}[Proof of Corollary~{\normalfont{\ref{sum08082012}}}]
  Assume towards a contradiction that, for some natural numbers $m>n$,
  there exists either an operator $R\colon C_0[0,\omega_1)^m\to
  C_0[0,\omega_1)^n$ which is bounded below, or an operator $T\colon
  C_0[0,\omega_1)^n\to C_0[0,\omega_1)^m$ which is surjective.  We
  shall focus on the first case; the other is very similar.  The proof
  is best explained if we represent the operator $R\colon
  C_0[0,\omega_1)^m\to C_0[0,\omega_1)^n$ by the operator-valued
  $(n\times m)$-matrix $(R_{j,k})_{j,k=1}^{n,m}$ given by $R_{j,k} =
  Q_j^{(n)}RJ_k^{(m)}\in\mathscr{B}(C_0[0,\omega_1))$, where
  $Q_j^{(n)}\colon C_0[0,\omega_1)^n\to C_0[0,\omega_1)$ and
  $J_k^{(m)}\colon C_0[0,\omega_1)\to C_0[0,\omega_1)^m$ denote the
  $j^{\text{th}}$ coordinate projection and $k^{\text{th}}$ coordinate
  embedding, respectively.

  Using elementary column operations, we can reduce the scalar-valued
  matrix $S = \bigl(\phi(R_{j,k})\bigr)_{j,k=1}^{n,m}$ to
  column-echelon form; that is, we can find an invertible,
  scalar-valued $(m\times m)$-matrix~$U$ such that $SU$ has
  column-echelon form. Since $m>n$, the final column of~$SU$ must be
  zero. Consequently, each operator in the final column of the
  matrix~$RU$ belongs to~$\mathscr{M} =
  \mathscr{HG}(C_0[0,\omega_1))$, so that
  $RUJ_m^{(m)}\in\mathscr{HG}(C_0[0,\omega_1),C_0[0,\omega_1)^n)$. This,
  however, contradicts Theorem~\ref{w1notUEC} because each of the
  operators $J_m^{(m)}$, $U$ and~$R$ is bounded below, and therefore
  the range of~$RUJ_n^{(n)}$ is isomorphic to its
  domain~$C_0[0,\omega_1)$.
\end{proof}

\begin{proof}[Proof of Corollary~{\normalfont{\ref{cor08082012b}}}]
  Since~$P$ is idempotent, we have $\phi(P)\in\{0,1\}$. We shall
  consider the case where $\phi(P) = 0$; the case where $\phi(P) = 1$
  is similar, just with $P$ and $I_{C_0[0,\omega_1)}-P$ interchanged.
  Let $X =\ker P$ and $Y =
  P[C_0[0,\omega_1)]$. Lemma~\ref{lemma25102012} implies that~$Y$ is
  contained in~$\ker P_D$ for some club subset~$D$
  of~$[0,\omega_1)$. By Corollary~\ref{OmegaXiLemma10Jan2013ii}, $\ker
  P_D$ is isomorphic to a complemented subspace of~$C_0(L_0)$, so that
  the same is true for~$Y$, say $C_0(L_0)\cong Y\oplus Z$ for some
  Banach space~$Z$.  Writing $c_0(\N,W)$ for the $c_0$-direct sum of
  countably many copies of a Banach space~$W$ and using
  Corollary~\ref{c0sumL0}, we obtain
  \[ C_0(L_0)\cong c_0(\N,C_0(L_0))\cong c_0(\N, Y\oplus Z)\cong
  Y\oplus c_0(\N, Y\oplus Z)\cong Y\oplus C_0(L_0). \] Consequently
  $C_0[0,\omega_1)\cong C_0[0,\omega_1)\oplus Y$ because
  $C_0[0,\omega_1)$ contains a complemented subspace isomorphic
  to~$C_0(L_0)$ by Corollary~\ref{OmegaXiLemma10Jan2013iii}.
  Theorem~\ref{surjective} implies that~$X$ contains a complemented
  subspace which is isomorphic to~$C_0[0,\omega_1)$, so that $X\cong
  W\oplus C_0[0,\omega_1)$ for some Banach space~$W$, and hence we
  have
  \[ X\cong W\oplus C_0[0,\omega_1) \cong W\oplus
  C_0[0,\omega_1)\oplus Y\cong X\oplus Y = C_0[0,\omega_1), \] as
  required.
\end{proof}

\begin{proof}[Proof of Corollary~{\normalfont{\ref{cor26Oct2012a}}}] 
  Let $U$ be an isomorphism of~$C_0[0,\omega_1)$ onto the closed
  subspace~$X$, and consider the operator $T =
  JU\in\mathscr{B}(C_0[0,\omega_1))$, where $J\colon X\to
  C_0[0,\omega_1)$ denotes the natural inclusion. Then~$T$ fixes a
  copy of~$C_0[0,\omega_1)$. Hence, by Theorem~\ref{thmcharloywillis},
  we can find operators~$R$ and~$S$ on~$C_0[0,\omega_1)$ such that
  $STR = I_{C_0[0,\omega_1)}$. This implies that $TR$ is an
  isomorphism of~$C_0[0,\omega_1)$ onto its range~$Y$, which is
  contained in~$X$, and~$TRS$ is a projection of~$C_0[0,\omega_1)$
  onto~$Y$.
\end{proof}

\begin{proof}[Proof of Corollary~{\normalfont{\ref{ogdenthm}}}]
  This proof follows closely that of \cite[Proposition~8]{willis},
  where any unexplained terminology can also be found.

  We begin by showing that Willis's ideal of compressible operators
  on~$C_0[0,\omega_1)$, as defined in \cite[p.~252]{willis}, is equal
  to the Loy--Willis ideal~$\mathscr{M}$. Indeed,
  \cite[Proposition~2]{willis} and Corollary~\ref{sum08082012} show
  that the identity operator on $C_0[0,\omega_1)$ is not compressible,
  so that the ideal of compressible operators is proper, and hence
  contained in~$\mathscr{M}$. Conversely, each operator
  $T\in\mathscr{M}$ is compressible by \cite[Proposition~1]{willis}
  because $T$ factors through a complemented subspace~$X$
  of~$C_0[0,\omega_1)$ such that $X\cong C_0(L_0)$, and
  Corollary~\ref{c0sumL0} implies that $X$ is isomorphic to the
  $c_0$-direct sum of countably many copies of itself.

  Next, we observe that null sequences in~$\mathscr{M}$ factor, in the
  sense that for each norm-null sequence~$(T_n)_{n\in\N}$
  in~$\mathscr{M}$, there are $T\in\mathscr{M}$ and a norm-null
  sequence~$(S_n)_{n\in\N}$ in~$\mathscr{M}$ such that $T_n = TS_n$
  for each $n\in\N$. This is a standard consequence of Cohen's
  Factorization Theorem (\emph{e.g.}, see
  \cite[Corollary~I.11.2]{bonsallduncan}), which applies
  because~$\mathscr{M}$ has a bounded left approximate
  identity. However, we do not need Cohen's Factorization Theorem to
  verify this. Indeed, for each $n\in\N$, take a club subset~$D_n$
  of~$[0,\omega_1)$ such that $T_n^*\delta_\alpha = 0$ for each
  $\alpha\in D_n$. Then $D = \bigcap_{n\in\N} D_n$ is a club subset
  of~$[0,\omega_1)$ such that $T_n^*\delta_\alpha = 0$ for each
  $\alpha\in D$ and $n\in\N$, and hence
  Corollary~\ref{OmegaXiLemma10Jan2013ii} and
  Lemma~\ref{lemma25102012} imply that $T = I_{C_0[0,\omega_1)} -
  P_D\in\mathscr{M}$ is an operator such that $T_n = TT_n$ for each
  $n\in\N$.

  Now consider an algebra homomorphism~$\theta$
  from~$\mathscr{B}(C_0[0,\omega_1))$ into some Banach
  algebra~$\mathscr{C}$. Then \cite[Proposition~7]{willis} implies
  that the continuity ideal of~$\theta|_{\mathscr{M}}$
  contains~$\mathscr{M}$, so that the mapping
  $S\mapsto\theta(TS),\,\mathscr{M}\to\mathscr{C},$ is continuous for
  each fixed~$T\in\mathscr{M}$. Since null sequences factor, as shown
  above, this proves the continuity of~$\theta|_{\mathscr{M}}$, and
  thus of~$\theta$ because~$\mathscr{M}$ has finite codimension
  in~$\mathscr{B}(C_0[0,\omega_1))$.
\end{proof}

\begin{proof}[Proof of Corollary~{\normalfont{\ref{commsandtraces}}}]
  Each operator in~$\mathscr{M}$ factors through the Banach space
  $C_0(L_0)$ by Theorem~\ref{thmcharloywillis}.
  Corollary~\ref{c0sumL0} states that $C_0(L_0)$ is isomorphic to the
  $c_0$-direct sum of countably many copies of itself, so that
  \cite[Proposition~3.7]{laustsen} implies that each operator
  on~$C_0(L_0)$ is the sum of at most two commutators, and therefore
  each operator which factors through $C_0(L_0)$ is the sum of at most
  three commutators by \cite[Lemma~4.5]{laustsen}.

  Suppose that $\tau$ is a trace on
  $\mathscr{B}(C_0[0,\omega_1))$. The first part of the proof implies
  that $\mathscr{M}\subseteq\ker\tau$, so that $\tau(T + \lambda
  I_{C_0[0,\omega_1)}) = \lambda\tau(I_{C_0[0,\omega_1)}) = \varphi(T
  + \lambda I_{C_0[0,\omega_1)}) \tau(I_{C_0[0,\omega_1)})$ for each
  $T\in\mathscr{M}$ and $\lambda\in\mathbb{K}$, as desired.  The
  converse implication is clear.
\end{proof}

\section*{Acknowledgements} 
\noindent
The authors would like to thank Richard Smith for a private exchange
of emails~\cite{smith}, which was the seed that eventually grew into
the Topological Dichotomy (Theorem~\ref{dichotomy}).

Part of this work was carried out during a visit of the second author
to Lancaster in February 2012, supported by a London Mathematical
Society Scheme 2 grant (ref.~21101). The authors gratefully
acknowledge this support. The second author was also partially
supported by the Polish National Science Center research grant
2011/01/B/ST1/00657.

\bibliographystyle{amsplain}

\bigskip
\end{document}